\documentclass[11pt]{amsart}
\usepackage{amsmath}
\usepackage{amsthm}
\usepackage{amssymb}
\usepackage{graphics}
\usepackage{bm}
\usepackage{setspace}

\usepackage{amsfonts}
\usepackage{latexsym}
\usepackage{mathdots}
\usepackage{mathrsfs}
\usepackage{enumitem}
\usepackage{eucal}

\usepackage{color}

\newcommand{\be}{\begin{equation}}
\newcommand{\ee}{\end{equation}}
\newcommand{\ba}{\begin{eqnarray}}
\newcommand{\ea}{\end{eqnarray}}
\newcommand{\bi}{\begin{itemize}}
\newcommand{\ei}{\end{itemize}}
\newcommand{\bn}{\begin{enumerate}}
\newcommand{\en}{\end{enumerate}}
\newcommand{\bbm}{\begin{bmatrix}}
\newcommand{\ebm}{\end{bmatrix}}
\newcommand{\bp}{\begin{proof}}
\newcommand{\ep}{\end{proof}}
\newcommand{\nn}{\nonumber}

\newcommand{\mr}{\ensuremath{\mathrm}}
\newcommand{\scr}{\ensuremath{\mathscr}}
\newcommand{\mbf}{\ensuremath{\mathbf}}
\newcommand{\mc}{\ensuremath{\mathcal}}

\newcommand{\ov}{\ensuremath{\overline}}
\newcommand{\sm}{\ensuremath{\setminus}}

\newcommand{\ga}{\ensuremath{\gamma}}
\newcommand{\Om}{\ensuremath{\Omega}}

\newcommand{\La}{\ensuremath{\Lambda }}
\newcommand{\la}{\ensuremath{\lambda }}

\def\C{\mathbb{C}}
\def\R{\mathbb{R}}
\def\D{\mathbb{D}}
\def\T{\mathbb{T}}

\def\N{\mathbb{N}}
\def\B{\mathbb{B}}
\def\A{\mathcal{A}}

\renewcommand{\H}{\ensuremath{\mathcal{H} }}
\newcommand{\J}{\ensuremath{\mathcal{J} }}

\newcommand{\m}{\ensuremath{\mathbf{m} }}
\newcommand{\n}{\ensuremath{\mathbf{n} }}

\newcommand{\K}{\ensuremath{\mathcal{K} }}
\renewcommand{\L}{\ensuremath{\mathcal{L} }}

\newcommand{\F}{\ensuremath{\mathbb{F} }}
\newcommand{\rft}{\ensuremath{\hat{\mc{F}} _R }}
\newcommand{\rct}{\ensuremath{\hat{\mc{C}} _R }}
\newcommand{\lft}{\ensuremath{\hat{\mc{F}} _L }}
\newcommand{\lct}{\ensuremath{\hat{\mc{C}} _L }}

\newcommand{\ip}[2]{\ensuremath{\left\langle {#1} , {#2} \right\rangle}}

\renewcommand{\dim}[1]{\ensuremath{\mathrm{dim} \left( {#1} \right) }}
\newcommand{\ran}[1]{\ensuremath{\mathrm{Ran} \left( {#1} \right) }}
\renewcommand{\ker}[1]{\ensuremath{\mathrm{Ker} \left( {#1} \right) }}

\newcommand{\im}[1]{\ensuremath{\mathrm{Im} \left( {#1} \right) }}
\newcommand{\re}[1]{\ensuremath{\mathrm{Re} \left( {#1} \right) }}

\numberwithin{equation}{section}
\numberwithin{subsection}{section}

\newtheorem{thm}[subsection]{Theorem}

\newtheorem{lemma}[subsection]{Lemma}
\newtheorem{prop}[subsection]{Proposition}
\newtheorem{cor}[subsection]{Corollary}

\theoremstyle{definition}
\newtheorem{defn}[subsection]{Definition}
\newtheorem{remark}[subsection]{Remark}
\newtheorem{eg}[subsection]{Example}

-.5truein

\advance\voffset by -.1truein

\title[Free Aleksandrov-Clark Theory]{Non-commutative Clark measures for the Free and Abelian Toeplitz Algebras}

\author{M.T. Jury}
\address{University of Florida}
\email{mjury@ad.ufl.edu}

\author{R.T.W. Martin}
\address{University of Cape Town}
\email{rtwmartin@gmail.com}

\begin{document}
\bibliographystyle{unsrt}
\maketitle
\onehalfspace

\begin{abstract}

   We construct a non-commutative Aleksandrov-Clark measure for any element in the operator-valued free Schur class, the closed unit ball of the free Toeplitz algebra of vector-valued full Fock space over $\C ^d$.
Here, the free (analytic) Toeplitz algebra is the unital weak operator topology (WOT)-closed algebra generated by the component operators of the free shift, the row isometry of left creation operators. This defines a bijection between the free operator-valued Schur class and completely positive maps (non-commutative AC measures) on the operator system of the free disk algebra, the norm-closed algebra generated by the free shift.

Identifying Drury-Arveson space with symmetric Fock space, we determine the relationship between the non-commutative AC measures for elements of the operator-valued commutative Schur class (the closed unit ball of the WOT-closed Toeplitz algebra generated by the Arveson shift) and the AC measures of their free liftings to the free Schur class.

\end{abstract}

\section{Introduction}

In the classical, single-variable theory of Hardy spaces of analytic functions in the complex unit disk, $\D$, there are natural bijections between the three classes of objects:
\bn
    \item the \emph{Schur class}, $\scr{S}$, of contractive analytic functions on the complex unit disk, $\D$,
    \item the \emph{Herglotz class}, $\scr{S} ^+$, of analytic functions with non-negative real part on the disk, and,
    \item the cone of positive finite Borel measures on the unit circle, $\T $.
\en
The bijection between Schur functions, $b$, and Herglotz functions, $H$, is given by:
$$ b \mapsto \frac{1+b}{1-b} \in \scr{S} ^+; \quad \quad \mbox{and} \quad \quad H \mapsto \frac{H-1}{H+1} \in \scr{S}, $$ these maps are compositional inverses (we assume here that $b$ is not the constant function $b \equiv 1$). The bijection between the second two sets is given by the Herglotz representation formula:
\be \frac{1+b(z)}{1-b(z)} =: H_b (z) = \int _\T \frac{1 +z\zeta ^*}{1-z\zeta ^*} \mu _b (d\zeta ) + i \im{H_b (0)}, \label{classicH} \ee (this is really a bijection modulo imaginary constants). In the above $\zeta ^* := \ov{\zeta }$ denotes complex conjugate. The unique positive Borel measure $\mu _b$ corresponding to $b$, is called the Herglotz or \emph{Aleksandrov-Clark measure} of $b$. More generally, for any $\alpha \in \T$, the Herglotz measure $\mu _\alpha := \mu _{b\alpha ^*}$ is called an Aleksandrov-Clark (AC) measure for $b$. The theory of Aleksandrov-Clark measures has played an important role in the development of Hardy space theory and model theory for contractions on Hilbert space, as well as in characterizations of the Schur class \cite{PolSar-AC,Saks-AC,Math-AC,GR-model,Clark1972}.

Given any AC measure, $\mu _b$, it is natural to consider the associated measure space $L^2 (\mu _b) := L^2 (\mu _b , \T )$ of measurable functions on the circle which are square-integrable with respect to $\mu _b$, as well as the analytic subspaces $H^2 (\mu _b ), H^2 _0 (\mu _b ) \subseteq L^2 (\mu _b)$,
$$ H^2 (\mu _b) := \bigvee _{n\geq 0} \zeta ^n \ \supseteq \ \bigvee _{n \geq 1} \zeta ^n=: H^2 _0 ( \mu _b), $$ the closed linear spans of the analytic polynomials and non-constant analytic monomials, respectively. A function-theoretic argument combined with the classical distance formula of Szeg\"{o}-Kolmomogoroff-Kre\v{\i}n for the distance from $H^2 _0 (\mu _b)$ to the constant function $1$ in $L^2 (\mu _b)$ shows that $H^2 _0 ( \mu _b) = H^2 (\mu _b) = L^2 ( \mu _b)$ if and only if $b$ is an extreme point of the Schur class \cite[Chapter 4, Chapter 9]{Hoff}.

On the other hand, given any contractive analytic function, $b$, on the open unit disk, it is also natural to consider the sesqui-analytic positive kernel function $k^b : \D \times \D \rightarrow \C$:
$$ k^b (z,w) := \frac{1 - b(z) b(w) ^*}{1-zw^*}; \quad \quad z,w \in \B ^d, $$ the \emph{deBranges-Rovnyak kernel} of $b$. Elementary reproducing kernel Hilbert space (RKHS) theory implies that there is a unique RKHS of analytic functions in the disk, $\H (k^b)$, corresponding to $k^b$, and that $\H (k^b )$ is contractively contained in the Hardy space $H^2 (\D )$. This space is called the \emph{deBranges-Rovnyak space} of $b$ and we will use the standard notation $\scr{H} (b) := \H (k^b)$. One can also show that, in this single-variable setting, any deBranges-Rovnyak space is invariant for $S^*$, the \emph{backward shift} on $H^2 (\D)$ which acts as the difference quotient: $$ (S^*h) (z) = \frac{h(z) -h(0)}{z}. $$ Here the shift, $S$, is the isometry of multiplication by $z$ on $H^2 (\D)$, and this operator is central to the study of function theory and operator theory on Hardy space \cite{Nik-Shift,NF,Hoff}.

In the seminal paper \cite{Clark1972}, D.N. Clark established the following results for the case of inner $b$ (the general versions for all Schur class functions can be found in \cite[Chapter III]{Sarason-dB}): Let $\hat{Z} ^b$ denote the unitary operator of multiplication by the independent variable in $L^2 (\mu _b)$. The analytic subspace $H^2 (\mu _b)$ is invariant for $\hat{Z} ^b$, and we set $Z^b := \hat{Z} ^b | _{H^2 (\mu _b )}$, an isometry which equals $\hat{Z} ^b$ if and only if $b$ is an extreme point of the Schur class.

\begin{lemma}{ (weighted Cauchy transform)} \label{weightedCT}
    For any contractive analytic $b \in \scr{S}$, and any $\alpha \in \T$, the \emph{weighted Cauchy transform} $\mc{F} _\alpha : H^2 (\mu _{b\alpha ^*}) \rightarrow \scr{H} (b \alpha ^* ) = \scr{H} (b)$ defined by
$$ (\mc{F} _\alpha f ) (z) := (1 - b(z) \alpha ^* ) \int _\T \frac{f(\zeta )}{1 -z\zeta ^* } \mu _{b \alpha ^* } (d\zeta ),  $$ is a linear isometry of the analytic subspace $H^2 (\mu _{b\alpha ^*} )$
onto the deBranges-Rovnyak space $\scr{H} (b)$.
\end{lemma}

For simplicity assume $b(0) = 0$ and let $X ^* := S^* | _{\scr{H} (b)}$. For any $\alpha \in T$, let $\mc{F} _\alpha := \mc{F} _{b\alpha ^*}$ and $Z_\alpha := Z^{b \alpha ^*}$.

\begin{thm}{ (Clark's unitary perturbations)} \label{uniperturb}
     Let $b \in \scr{S}$ be a contractive analytic function in the disk (assume $b(0) =0$). Given any $\alpha \in \T$,  the weighted Cauchy transform $\mc{F} _\alpha$ intertwines the co-isometry $Z _\alpha ^*$ with a rank-one perturbation of $X^*$:
$$ X _\alpha ^* := \mc{F} _\alpha Z _\alpha  ^* \mc{F} _\alpha ^* = X^* + \ip{\cdot}{1} S^* b \alpha ^*. $$
The point evaluation vector at $0$, $k_0 ^b \equiv 1 \in \scr{H} (b)$ is cyclic for each $X _\alpha$.

If $b$ is an extreme point of the Schur class then $Z _\alpha = \hat{Z} ^{b\alpha ^*} $ is unitary so that each $X _\alpha$ is a rank-one unitary perturbation of the restricted backward shift $X$. In this case if
$P _\alpha$ denotes the projection-valued measure of $X _\alpha$ then $\mu _\alpha ( \Om ) = \ip{P _\alpha (\Om ) 1}{ 1 }.$
\end{thm}

\begin{remark}
    In the case where $b$ is an extreme point (so that $H^2 (\mu _b ) = L^2 (\mu _b)$), the inverse of the weighted Cauchy transform $\mc{F} _\alpha$ implements a spectral realization for the unitary operator $X _\alpha$.
\end{remark}

Recently, the concept of Aleksandrov-Clark measure and all of the above results have been generalized to the several-variable setting of Drury-Arveson space \cite{Jur2014AC} (see \cite{JM} for the vector-valued version). Here, the Drury-Arveson space, $H^2 _d$, consists of analytic functions on the open unit ball of $d$-dimensional complex space, and is a canonical several-variable generalization of the classical Hardy space $H^2 (\D )$. We will briefly recall the relevant definitions in the upcoming subsection. The appropriate several-variable analogue of the Schur class is the closed unit ball of the several-variable (analytic) \emph{Toeplitz} or \emph{Hardy algebra}, $H^\infty _d $, the (commutative) WOT-closed operator algebra generated by the Arveson $d-$shift on $H^2 _d$. (Here, note that the classical Schur class of the disk can be identified with the closed unit ball of the Banach algebra $H^\infty (\D ) = H^\infty _1$ of bounded analytic functions in the open disk, and that $H^\infty (\D )$ can be identified with the unital WOT-closed operator algebra generated by the shift.) The Aleksandrov-Clark measures are necessarily promoted to positive linear functionals (or completely positive maps in the vector-valued setting) acting on a certain `symmetrized' operator subsystem $\mc{S} + \mc{S} ^* $, $\mc{S} := \mc{S} _d$, of $\mc{A} + \mc{A}  ^*$, where $\mc{A}:= \A _d $ is the left free disk algebra, the unital norm-closed (non-commutative) operator algebra generated by the left creation operators on the full Fock space over $\C ^d$. The measure space $H^2 (\mu _b)$ in the several-variable setting is naturally generalized to a Gelfand-Naimark-Segal-type space with inner product constructed using the non-commutative AC measure, $\mu _b$, of $b$ (as in the proof of Stinespring's dilation theorem from \cite{Paulsen}). With this dictionary, the classical correspondence between the Schur class, Herglotz functions and AC measures can be extended to define bijections between \cite{Jur2014AC,JM}:
\bn
\item The (operator-valued, several-variable) Schur class, $\scr{S} _d (\H) := [H^\infty _d \otimes \L (\H) ] _1 $,
\item  The (operator-valued, several-variable) Herglotz-Schur class, $\scr{S} _d ^+ (\H)$, consisting of Herglotz-Schur functions $H_b (z) := (I-b(z) ) ^{-1} (I+b(z)$, on $\B ^d$, for $b \in \scr{S} _d (\H )$, and,
\item   The positive cone $CP (\mc{S}  ; \H)$ of all completely positive (CP) operator-valued maps $\mu $ from the symmetrized operator system $\mc{S} + \mc{S}  ^* $ into $ \L (\H )$.
\en
As before, if $b \in \scr{S} _d (\H )$, the corresponding CP map $\mu _b \in CP (\mc{S} ; \H )$ is called the Aleksandrov-Clark (AC) map, or non-commutative AC measure, of $b$. These AC maps are direct several-variable generalizations of the classical AC measures.

In this paper our goal is two-fold. Our first aim is to further extend the notion of a non-commutative Aleksandrov-Clark measure, the above bijection between the Schur class and AC measures, Clark's unitary perturbations  and several related results to the setting of the (left and right)
\emph{free Schur class} of the (left and right) \emph{free analytic Toeplitz algebra}. Here, the left (right) non-commutative or \emph{free analytic Toeplitz algebra}, or more simply \emph{free Toeplitz algebra}, $L^\infty _d \ (R^\infty _d)$,  is the unital WOT-closed algebra generated by the left (right) creation operators on the full Fock space, $F^2 _d$, over $\C ^d$. As in the abelian case, we will often omit the term analytic and call $L^\infty _d$ the left free Toeplitz algebra. The left and right (operator-valued) free Schur classes, $\scr{L}  _d (\H)$, $\scr{R}  _d (\H ) $ are then the closed unit balls of the left and right free (operator-valued) Toeplitz algebras associated to vector-valued Fock space $F^2 _d \otimes \H$. The connection with the commutative theory is that Drury-Arveson space, $H^2 _d$, can be naturally identified with symmetric Fock space, $H^2 _d \subset F^2 _d$, and under this identification $H^2 _d$ is co-invariant and full (\emph{i.e.} cyclic) for both the left and right \emph{free shifts} (the row isometries of left and right creation operators). That is, if $L$ denotes the left free shift, $L$ is the minimal row isometric dilation of its compression to $H^2 _d$, and this compression is the commutative Arveson $d-$shift, $S$, on $H^2 _d$. The commutative several-variable Toeplitz algebra, $H^\infty _d$, can then be identified with the quotient of either the left or right free Toeplitz algebra by the two-sided commutator ideal. Equivalently, $H^\infty _d$ can be obtained as the compression of $L^\infty _d $ or $R^\infty _d$ to symmetric Fock space $H^2 _d$, and this compression is a completely contractive unital epimorphism \cite{DP-NP}. By commutant lifting, given any commutative Schur class element $b \in \scr{S} _d (\H )$, there are both left and right \emph{free lifts}, $B^L \in \scr{L} _d (\H )$, $B^R \in \scr{R} _d (\H )$ so that their image under the quotient map is $b$ \cite{Ball2001-lift,Dav2011}. That is, if, for example, $M^L _B$ denotes left multiplication by $B ^L$ on $F^2 _d$, then $(M^L _B) ^* | _{H^2 _d} = M_b ^*$, and $B^L, B^R, b$ have the same norm. Of course these free lifts need not be unique. We will see that left and right free lifts come in pairs $B:= (B^L, B^R)$ which are conjugate via transposition, the canonical involution between the left and right free Toeplitz algebras, and that each pair, $B$, corresponds to a unique non-commutative Aleksandrov-Clark measure. This non-commutative AC measure is a completely positive (CP) map, $\mu _B : \mc{A} + \mc{A} ^* \rightarrow \L (\H)$.

Our second goal, then, is to relate any non-commutative AC completely positive measure $\mu _B$ of a transpose-conjugate pair of free lifts $B = (B ^L, B ^R) \in \scr{L} _d (\H ) \times \scr{R} _d (\H )$ of a given commutative $b \in \scr{S} _d (\H)$ with the AC map $\mu _b $ acting on the symmetrized subsystem $\mc{S}  + \mc{S} ^* \subseteq \mc{A} + \mc{A}  ^*$ as constructed in \cite{Jur2014AC,JM}. In particular, we will show that any such $\mu _B$ is a completely positive extension of $\mu _b$, and that $b$ has a unique pair of free lifts if and only if $\mu _b, \mu _B$ are \emph{quasi-extreme} in the sense of \cite{Jur2014AC,JM}, a property which reduces to the classical Szeg\"{o} approximation property: $H^2 (\mu _b) = H^2 _0 (\mu _b)$ in the single-variable, scalar-valued setting (and which is equivalent to being an extreme point in this case) \cite{JMqe,Hoff}. This bijective characterization of the set of all free lifts of a given Schur class $b \in \scr{S} _d (\H)$ provides an alternative to the canonical deBranges-Rovnyak colligation and transfer function realization of the commutative and free Schur classes of \cite{Ball2008,Ball2010,Ball2006Fock}. In particular our characterization has the advantage of providing a bijective parametrization of the set of all (generally non-unique in the commutative case) canonical deBranges-Rovnyak colligations in terms of certain completely positive extensions of the AC map $\mu _b$  to the full free disk operator system (equivalently in terms of certain free lifts of $b$).  In Section \ref{TFsection} we work out the precise relationship between the canonical deBranges-Rovnyak colligations in the free and commutative settings.

\subsection{Preliminaries} \label{prelim}

Recall that Drury-Arveson space, $H^2 _d$, is the unique RKHS on $\B ^d:= (\C^d ) _1$ corresponding to the several-variable sesqui-analytic Szeg\"{o} kernel:
$$ k (z,w) := \frac{1}{1-zw^*}; \quad \quad z, w \in \B^d, $$ where in the above $z = (z_1, ... , z_d ); \ w^* := (\ov{w _1}, ... , \ov{w_d} )$, and $zw^* := (w, z) _{\C ^d}$ (all inner products are assumed conjugate linear in the first argument).

Given any RKHS $\H (K)$ of $\H$-valued functions on a set $X$, a natural construction to consider is the \emph{multiplier algebra}, $\mr{Mult} (\H (\K) )$, of $\H (K)$. This is the algebra of all functions $m : X \rightarrow \L (\H)$ so that $mf \in \H (K)$ for all $f \in \H (K)$. That is, the multiplier algebra is the algebra of all functions, or \emph{multipliers}, which multiply $\H (K)$ into itself. This algebra is clearly unital, and standard functional analytic arguments show that any multiplier, $F$, defines a bounded linear multiplication operator, $M _F$, on $\H (K)$ and under this identification, $\mr{Mult} (\H (K) )$ is closed in the weak operator topology (WOT) of $\L (\H (K) )$.

The multiplier algebra, $\mr{Mult} (H ^2 _d )$, of the RKHS $H^2 _d$ is $H^\infty _d$, the several-variable (analytic) Toeplitz or Hardy algebra, the WOT-closure of the unital operator algebra generated by the Arveson $d$-shift. Here recall the Arveson shift, $S : H^2 _d \otimes \C ^d \rightarrow H^2 _d$, is the (commutative) row partial isometry $S = (S_1, ... , S_d)$ whose component operators act as multiplication by the independent variables: $$ (S_j h ) (z) = z_j h(z) = z_j h (z_1 , ... , z_d) ; \quad \quad 1 \leq j \leq d, \quad h \in H^2 _d. $$

The several-variable \emph{Schur class}, $\scr{S} _d = \scr{S} _d (\C )$, is the closed unit ball of this multiplier algebra. More generally the operator-valued Schur classes are the closed unit balls of the multipliers between vector-valued Drury-Arveson spaces:
$$ \scr{S} _d (\H _1, \H_2 ) := [ \mr{Mult} (H^2 _d \otimes \H_1 , H^2 _d \otimes \H _2 ) ] _1.$$ We will focus on the `square' case where $\H _1 = \H _2 = \H$:
$$ \scr{S} _d (\H) := \scr{S} _d (\H , \H) = [ H^\infty _d \otimes \L (\H ) ] _1, $$ our results can be easily extended to the general rectangular setting (the rectangular Schur classes can be embedded in square Schur classes by adding rows or columns of zeroes).

Given any $b \in \scr{S} _d (\H)$ (or more generally any $b \in \scr{S} _d (\H _1 , \H _2 ) = \left[ H^\infty _d \otimes \L (\H _1 , \H _2 ) \right] _1$), one can construct the positive \emph{deBranges-Rovnyak kernel},
$$ k^b (z,w) := \frac{ I - b(z) b(w) ^*}{1-zw^*}; \quad \quad z,w \in \B^d, $$ and the associated \emph{deBranges-Rovnyak RKHS}, $\scr{H} (b)$. By standard RKHS theory, these spaces are always contractively contained in $H^2 _d \otimes \H$.

It will often be convenient to view $H^2 _d$ as symmetric or bosonic Fock space over $\C ^d$ \cite[Section 4.5]{Sha2013}: First recall that the full Fock space over $\C ^d$, $F^2 _d$, is the direct sum of all tensor powers of $\C ^d$:
\ba F^2 _d & := & \C \oplus \left( \C ^d \otimes \C ^d \right) \oplus \left( \C ^d \otimes \C ^d \otimes \C ^d \right) \oplus \cdots \nn \\
& =& \bigoplus _{k=0} ^\infty \left( \C ^d \right) ^{k \cdot \otimes }. \nn \ea Fix an orthonormal basis $e_1, ... , e_d$ of $\C ^d$. The left creation
operators $L_1, ..., L_d$ are the operators which act as tensoring on the left by these basis vectors:
$$ L_k f := e_k \otimes f; \quad \quad f \in F^2 _d, $$ and similarly the right creation operators $R_k; \ 1\leq k \leq d$ are defined by tensoring on the right $$ R_k f := f \otimes e_k. $$ The left and right free shifts are the row operators $L := (L_1 , ... , L_d)$ and $R := (R_1, ... , R _d )$ which map
$F^2 _d \otimes \C ^d$ into $F^2 _d$. Both $L, R$ are in fact row isometries: $L ^* L = I_{F^2} \otimes I_d = R^* R$.  The orthogonal complement of the range of $L$ or $R$ is the vacuum vector $1$ which spans the the subspace $\C =: (\C^d) ^{0\cdot \otimes} \subset F^2 _d$. A canonical orthonormal basis for $F^2 _d$ is then $\{ e_\alpha \} _{\alpha \in \F ^d}$ where $e_\alpha = L^\alpha 1 = R^\alpha 1$ and $\F ^d$ is the free unital semigroup or monoid on $d$ letters.

Recall here that the free semigroup, $\F^d$, on $d \in \N$ letters, is the multiplicative semigroup of all finite products or \emph{words} in
the $d$ letters $\{1, ... , d \}$. That is, given words $\alpha := i_1 ... i_n$, $\beta := j_1 ... j_m$, $i_k, j_l \in \{1 , ... , d \}; \ 1\leq k \leq n, \ 1 \leq l \leq m$, their product $\alpha \beta $ is defined by concatenation:
$$ \alpha \beta = i_1 ... i_nj_1 ... j_m, $$ and the unit is the empty word, $\emptyset$, containing no letters. Given $\alpha = i_1 \cdots i_n$, we use the standard notation $|\alpha | = n$ for the length of the word $\alpha$.

For any permutation $\sigma$ on $n$ letters, one can define a unitary operator $U_\sigma$ on $(\C ^d) \otimes \C ^n \subset F^2 _d$ by
$$ U_\sigma (u_1 \otimes u_2 \otimes \cdots \otimes u_n ) := u_{\sigma (1) } \otimes \cdots \otimes u_{\sigma (n)}; \quad \quad u_k \in \C ^d. $$ This defines a representation, $\pi _n : \mr{Sym} (n) \rightarrow \L ( F^2 _d )$
of the symmetric or permutation group $\mr{Sym} (n)$ on $n$ letters.
The $n$th symmetric tensor product of $\C ^d$, $$ ( \C ^d ) ^n  = \frac{ \C ^d \otimes \C ^n }{\mr{Sym} (n)} \subset \C ^d \otimes \C ^n, $$ is (defined to be) the subspace of all common fixed points of the unitaries $U_\sigma$. The symmetric Fock space, $\mr{Sym} (F^2 _d )$ (we will shortly identify this with $H^2 _d$) is then the direct sum of all symmetric tensor products:
$$ \mr{Sym} (F^2 _d ) := \bigoplus _{n=0} ^{\infty} (\C ^d )^n. $$

Let $\N ^d$ be the unital additive semigroup or monoid of $d$-tuples of non-negative integers. By the universality property of the free unital semigroup $\F ^d$, there is a unital semigroup epimorphism $\la : (\F ^d , \cdot ) \rightarrow (\N ^d , +)$, the \emph{letter counting map} which sends a given word $\alpha = i_1 \cdots i_n \in \F ^d $ to $\mbf{n} = (n_1, ... ,n_d) \in  \N ^d$ where $n_k$ is the number of times the letter $k$ appears in the word $\alpha$.  For any $\n \in \N ^d$, we define the symmetric monomial
$$ L^\n := \sum _{\alpha \in \F ^d;  \  \la (\alpha ) = \n } L ^\alpha, $$ and it is then not difficult to verify that $\{ e_\n := L^\n 1 = R ^\n 1 \}$ is an orthogonal basis for $\mr{Sym} (F^2 _d )$ such that
$$ \ip{L^\n 1}{L^\m 1 } _{F^2} = \delta _{\n , \m} \frac{ | \n | !}{\n !}. $$ Here, and throughout, we use the standard notations $| \n | := n_1 + ... + n_d$ for $\n \in \N ^d$, and $\n ! := n_1 ! \cdots n_d !$. As shown in, \emph{e.g} \cite{Jur2014AC},
$$ \mc{S} _d := \bigvee _{\n \in \N ^d } L^\n = \bigvee _{z \in \B ^d} (I -L z^* ) ^{-1}, $$ where $\bigvee$ denotes norm-closed linear span and
$Lz ^* := L_1 \ov{z_1} + ... + L_d \ov{z_d}$. It follows that
$$ \mr{Sym} (F^2 _d ) = \bigvee (I -L z^* ) ^{-1} 1, $$ and it is easily verified that the map
$$ (I - Lz ^* ) ^{-1} 1 \mapsto k_z, $$ is an onto isometry which sends $e _\n = L^\n 1$ to $z^\n \frac{|\n | !}{\n !} \in H^2 _d$, where $z^\n := z_1 ^{n_1}\cdots z_d ^{n_d}$. For the remainder of the paper we will identify these two spaces and simply write $H^2 _d \subset F^2 _d$ for symmetric Fock space.

\section{Free Formal reproducing kernel Hilbert spaces}

It will be useful to review the theory of free or non-commutative (NC) formal reproducing kernel Hilbert spaces (RKHS), as introduced in \cite{Ball2003rkhs,Ball2006Fock}. This will allow us to define left and right free analogues of the commutative several-variable deBranges-Rovnyak spaces
$\scr{H} (b)$ associated to any $b \in \scr{S} _d (\H)$. If $B$ is a left or right free lift of $b$, we will see that there is very nice relationship and natural maps between the corresponding free and commutative deBranges-Rovnyak spaces. Moreover the left or right deBranges-Rovnyak space of $B$ will have a structure which is formally very similar to a commutative deBranges-Rovnyak RKHS, and it will be fruitful to exploit this analogy with the commutative setting to obtain a non-commutative or `free' extension of the Aleksandrov-Clark theory for the abelian Schur class $\scr{S} _d (\H)$ developed in \cite{Jur2014AC,JM}.

Any formal RKHS in the sense of \cite{Ball2003rkhs} is essentially a classical RKHS on a finitely generated unital semigroup (or monoid), $\mathbf{M} ^d$ (with $d$ generators), where the reproducing kernel $K _{\varphi, \vartheta};  \ \varphi, \vartheta \in \mbf{M} ^d$ is viewed as the formal power series coefficients of a `formal reproducing kernel' in two formal variables. The key difference between classical RKHS theory over finitely generated monoids and formal RKHS theory is the shift in focus from multipliers to \emph{formal multipliers}: Given a discrete RKHS, $\H (K _{\varphi , \vartheta})$, of functions on a finitely generated monoid $\mbf{M} ^d$, instead of the usual multiplier algebra, one can consider the \emph{convolution algebra} of bounded convolution operators from $\H (K _{\varphi, \vartheta})$ into itself. If one identifies elements of the discrete RKHS $\H (K _{\varphi , \vartheta})$ with formal power series indexed by $\mbf{M} ^d$, this convolution algebra can be viewed as the \emph{formal multiplier algebra}, the algebra of formal power series which multiply the formal RKHS into itself. We will primarily be interested in the case of $\F ^d$, the free unital semigroup on $d$ generators (the universal monoid on $d$ generators).

\subsection{Formal RKHS over $\F ^d$}

Let $\H$ be an auxiliary `coefficient' Hilbert space. We will call any positive kernel function $c : \F ^d \times \F ^d \rightarrow \L (\H )$ an operator-valued \emph{free coefficient kernel}, and the associated formal power series $$ K (Z,W) := \sum _{\alpha , \beta \in \F ^d  } Z^\alpha (W^*)^{\beta ^T} c (\alpha , \beta )  \in \L (\H ) \{ Z , W^* \}, $$ is called a (positive) \emph{free kernel}. Here $Z = (Z_1 , ... , Z_d)$ and $W^* = (W_1 ^* , ..., W_d ^* )$ are two sets of free (non-commuting) variables and given a word $\alpha = i_1 i_2 ... i_n ; \quad i_k \in \{ 1 , ... , d \},$ the transpose of $\alpha$ is $\alpha ^T = i_n ... i_1$. In the above, we have also used the notation
$\L (\H ) \{ Z , W^* \}$ for the linear space of all formal power series in the free variables $Z, W^*$ with coefficients in $\L (\H )$, and we will write $K_{\alpha , \beta } := c (\alpha , \beta )$ for the coefficient kernel corresponding to a free kernel $K$.

A Hilbert space $\K$ is called a \emph{free RKHS} of $\H$-valued functions if any $F$ in $\K$ can be written as a formal power series
$$ F (Z)  = \sum _{\alpha \in \F ^d  } Z^\alpha F_\alpha  \in \H \{ Z \}; \quad \quad F_\alpha \in \H $$ in the free variable $Z$, and if for each $\alpha \in \F ^d $ the linear \emph{coefficient evaluation map} $$ K_\alpha ^* \in \L (\K , \H ), $$ defined by
$$ K_\alpha ^* F := F_\alpha, $$ is bounded.  We write $K_\alpha \in \L (\H , \K)$ for the Hilbert space adjoint of this linear map.  The free coefficient kernel for $\K$ is defined by the coefficients
$$ K_{\alpha , \beta } := K_\alpha ^* K_\beta \in \L (\H). $$ The expression,
$$ K(Z,W) := \sum _{\alpha, \beta \in \F ^d} Z^\alpha (W ^* )^{\beta ^T} K_{\alpha ,\beta }, $$ defines a positive free kernel, called the \emph{free reproducing kernel} of $\K$. That is, $K_{\alpha, \beta}$ is necessarily a positive kernel function in the classical sense on the discrete set $\F ^d$, and classical RKHS theory implies that there is a bijection between free kernel functions $K$  and free RKHS, $\H (K)$, of free formal power series with free reproducing kernels $K$.  We write $\K := \mc{F} (K)$ if $\K$ is a free RKHS with free kernel $K$. Note that for any $\beta \in \F ^d$,
$$ K_\beta (Z) = \sum _{\alpha \in \F ^d } Z^\alpha K_{ \alpha , \beta}  \in  \L (\H , \mc{F} (\K) ) \{ Z \}. $$

If $\mc{F} (K)$ is a free RKHS of $\H$-valued functions (whose elements can be written as free formal power series), we can define formal \emph{point evaluation maps}:
$$ K_W ^* := \sum _{\alpha \in \F ^d } W  ^\alpha K_\alpha  ^* \in \L (\mc{F} (K)  , \H ) \{ W \}. $$
Also define the formal adjoint of free power series termwise as:
\be \label{freeadj} K_W := (K_W ^*)^* := \sum _{\alpha \in \F ^d } (W^*) ^{\alpha ^T} K_\alpha   \in \L (\H , \mc{F} (K ) ) \{ W ^* \}. \ee

Then for any $F \in \mc{F} (K)$, $K_W ^* F$ is defined termwise as
\ba K_W ^* F &= & \sum _{\alpha  \in \F ^d } W^\alpha K_\alpha ^* F   \nn \\
& = & \sum _{\alpha  \in \F ^d } W^\alpha F _\alpha \nn \\
& = & F (W), \nn \ea and
\ba K_Z ^* K_W & = & \sum _{\alpha } \sum _{\beta } Z^\alpha (W^* ) ^{\beta ^T} K_\alpha ^* K_\beta  \nn \\
& = &    \sum _{\alpha } \sum _{\beta } Z^\alpha (W^* ) ^{\beta ^T} K_{ \alpha , \beta } \nn \\
& = & K (Z,W). \ea These properties are formally analogous to properties of classical RKHS, and in many calculations it will be easier to
work with the formal point evaluation maps $K_W \in \L (\H , \mc{F} (K) ) \{ W^* \}$ in place of the bounded linear coefficient evaluation maps $K_\alpha \in \L (\H , \mc{F} (K) )$.

\begin{remark} \label{freeRKHS}
Up to this point, no new theory has been introduced. Under the identification of elements of a free RKHS of formal power series with their power series coefficients indexed by the free monoid $\F ^d$, the concept of a free RKHS is equivalent to that of a classical RKHS over $\F ^d$.
\end{remark}

As in classical RKHS theory, given any free RKHS $\mc{F} (K)$ of $\H-$valued free power series, there are naturally associated (formal) free (left and right) multiplier algebras. The noncommutativity of the unital free semigroup $(\F ^d , \cdot )$ leads to two different notions of formal multipliers: left multipliers and right multipliers (equivalently left or right convolution operators).

A bounded linear map $M : \mc{F} (k ) \rightarrow \mc{F} (K)$ between two free RKHS of $\H$ and $\J$-valued functions, respectively, is called a \emph{left free multiplier} if there is a formal power series
$$ M(Z) := \sum _{\alpha \in \F ^d } Z^\alpha M_\alpha \in \L (\H , \J ) \{ Z \}, $$ so that $M$ acts as left multiplication by $M(Z)$: For any $F \in \mc{F} (k)$,
\ba  (MF) (Z) & = &  M(Z) F(Z) = \left( \sum _{\alpha \in \F ^d } Z^\alpha M_\alpha \right) \left( \sum _\beta Z^\beta F_\beta \right)  \nn \\
& := & \sum _{\alpha, \beta } Z^{\alpha \beta } M_\alpha F_\beta \nn \\
& = & \sum _\ga Z^\ga \left( \sum _{\alpha \beta = \ga } M_\alpha F_\beta \right). \nn \ea  Similarly it is called a right multiplier if it acts as right multiplication by $M(Z)$:
$$ (MF) (Z) = M(Z) \bullet_R F (Z). $$ The above right product of formal power series is defined as
\ba M(Z) \bullet_R  F (Z) &=& \left( \sum _\alpha Z^\alpha M _\alpha \right) \bullet_R \left( \sum _\beta Z^\beta F_\beta \right) \nn \\
& := & \sum _{\alpha , \beta} Z^{\beta \alpha } M_\alpha F_\beta \nn \\
& = & \sum _{\ga } Z^\ga \left( \sum _{\beta \alpha = \ga } M_\alpha F_\beta \right). \nn \ea
The above shows that left and right formal free multiplication can be defined in terms of (left or right) convolution of the coefficients:

\begin{lemma}
    If a bounded linear $M : \mc{F} (k ) \rightarrow \mc{F} (K)$ acts as left or right multiplication by $M(Z) = \sum _\alpha Z^\alpha M_\alpha \in \L (\H , \J) \{ Z \}$  then
$$ M^* K_\alpha = \sum _{\beta \cdot \ga = \alpha } k_\ga  M_\beta ^* \in \L (\H ; \mc{F} (k) ), $$ or
$$ M^* K_\alpha  = \sum _{\ga \cdot \beta = \alpha } k_\ga  M_\beta ^* \in \L (\H ; \mc{F} (k) ), \quad \quad \mbox{respectively.} $$
\end{lemma}

The restatement of the above in terms of the formal point evaluation maps is again more formally analogous to the classical theory:

\begin{lemma}
    If $M : \mc{F} (k) \rightarrow \mc{F} (K)$ is a bounded left multiplier then
$$ M^* K_Z = k_Z M(Z) ^* \in \L (\H , \mc{F} (k) ) \{ Z^* \}.  $$ If it is a bounded right multiplier then
$$ M^* K_Z = k_Z \bullet_R M(Z) ^*  \in \L (\H , \mc{F} (k) ) \{ Z^* \}. $$
\end{lemma}

Analogues of classical RKHS results include:

\begin{thm}
    A formal power series $M(Z) = \sum _\alpha Z^\alpha M_\alpha  \in \L (\H , \J ) \{ Z \}$ defines a bounded left free multiplier from
$\mc{F} (k) $ into $\mc{F} (K)$ if and only if there is a $B > 0 $ so that
$$ \sum _{\beta ' \cdot \ga ' = \alpha ' } \sum _{ \beta \cdot \ga = \alpha } M_\beta k_{\ga , \ga ' } M_{\beta '} ^* \leq B \ K_{\alpha , \alpha '}, $$
as positive free coefficient kernels.

In particular, $\mc{F} (k)$ is contractively contained in $\mc{F} (K)$ if and only if $K_{\alpha , \beta} - k_{\alpha ,\beta}$ is a positive free coefficient kernel. Equality holds in the above with $B=1$ if and only if $M$ is a co-isometric left multiplier.
\end{thm}

The same statements hold for right free multipliers if one reverses the order of the products of the free semigroup elements $\beta, \ga$ and $\beta ' , \ga '$. Again, this can be restated in terms of formal point evaluation maps and free kernels:

\begin{thm}
 A formal power series $M(Z) = \sum _\alpha Z^\alpha M_\alpha  \in \L (\H , \J ) \{ Z \}$ defines a bounded left free multiplier from
$\mc{F} (k) $ into $\mc{F} (K)$ if and only if there is a $B > 0 $ so that
$$ M(Z) k (Z,W) M(W) ^* \leq B \ K(Z,W), $$ as free formal positive kernels.

Similarly it defines a bounded right free multiplier if and only if
$$ M(Z) \bullet_R k(Z,W) \bullet_R M(W) ^* \leq B \ K(Z,W). $$

In either case (right or left) multiplication by $M(Z)$ is a co-isometry if and only if equality holds with $B =1$ and $\mc{F} (k)$ is contractively contained in $\mc{F} (K)$ if and only if
$K - k$ is a positive free kernel.
\end{thm}

Given two free RKHS, $\mc{F} (k), \mc{F} (K)$, we define the left and right \emph{free multiplier spaces}, $\mr{Mult} ^L (\mc{F} (k) , \mc{F} (K) )$, $\mr{Mult} ^R (\mc{F} (k), \mc{F} (K ) )$, as the spaces
of all left and right free multipliers of $\mc{F} (k)$ into $\mc{F} (K)$.  As in the classical, commutative theory, any left (right) free multiplier, $F$, defines a bounded linear multiplication map, $M_F ^L : \mc{F} (k) \rightarrow \mc{F} (K)$ (or $M_F ^R$ in the right case), and under this identification, these multiplier spaces are $WOT$-closed. In the case where $\mc{F} (K) = \mc{F} (k)$, we write $\mr{Mult} ^L (\mc{F} (K)) := \mr{Mult} ^L (\mc{F} (K), \mc{F} (K) )$, for the unital
\emph{free left multiplier algebra} of $\mc{F} (K)$ (and similarly for the free right multiplier algebra). As observed above, the free left and right multiplier algebras of a free RKHS $\mc{F} (K)$ can be equivalently viewed as (what could be called) the free left and right \emph{convolution algebras} of the discrete
classical RKHS $\mc{H} (K_{\alpha , \beta} )$ corresponding to the free coefficient kernel $K _{\alpha , \beta}$ on $\F ^d \times \F ^d$.

Our main motivation for considering the theory of free formal RKHS is to apply it to the setting of the full Fock space, $F^2 _d$, over $\C ^d$. The example below (from \cite{Ball2003rkhs}) shows that the full Fock space can be naturally viewed as a free RKHS, the \emph{free Hardy space} over $d$ free variables. The WOT-closed unital operator algebras generated by the left and right creation operators, \emph{i.e.} the left and right free Toeplitz algebras, are then naturally identified with the left and right free multiplier algebras of this free RKHS.
\begin{eg}{ The full Fock space and the free Szeg\"{o} kernel.} \label{Freeg}

Any element $f \in F^2 _d$ has the form
$$ f = \sum _{\alpha \in \F ^d } f_\alpha L^\alpha 1; \quad \quad f_\alpha \in \C, $$ where $1$ denotes the vacuum vector and
$L$ is the left creation isometry. We can identify $f$ with the formal power series
$$ f(Z) := \sum _\alpha Z^\alpha f_\alpha. $$  Since $ f_\alpha = \ip{L^\alpha 1}{f} _{F^2},$ the coefficient evaluation vector $\hat{k} _\alpha$ is simply $\hat{k} _\alpha (Z) = Z^\alpha,$ and the free coefficient kernel is:
$$ \hat{k} _{\alpha , \beta} := \ip{\hat{k} _\alpha}{\hat{k}_\beta  } _{F^2} = \delta _{\alpha, \beta}. $$ The corresponding free kernel is then:
\ba \hat{k} (Z,W) & = & \sum _{\alpha , \beta \in \F ^d}  Z^\alpha (W^*)^{\beta ^T} \hat{k} _{\alpha, \beta} \nn \\
& = & \sum _{\alpha , \beta \in \F ^d}  Z^\alpha (W^*)^{\beta ^T} \delta _{\alpha ,\beta} \nn \\
& = & \sum _{\alpha \in \F ^d } Z^\alpha (W^*)^{\alpha ^T}. \nn \ea

This is a free analogue of the Szeg\"{o} kernel for Drury-Arveson space:
Indeed, replacing $Z, W^*$ with the commutative variables $z,w^* \in \B ^d$ yields:
\ba \hat{k} (z,w) & = & \sum _{\alpha \in \F ^d} z^\alpha (w^*) ^\beta \nn \\
& = & \sum _{\n \in \N ^d } \frac{| \n | ! }{\n !} z^\n (w^*) ^\n \nn \\
& = & \frac{1}{1-zw^*} \nn \\
& = & k (z,w), \nn \ea the Szeg\"{o} kernel for Drury-Arveson space. It makes sense to view $F^2 _d$ as the `free' Drury-Arveson space or free several-variable Hardy space.
\end{eg}

\subsection{Free deBranges-Rovnyak spaces}

Viewing $F^2_d$ or vector-valued $F^2 _d \otimes \H$ as a free RKHS, the left and right free Toeplitz algebras, $L^\infty _d$ and $R^\infty _d$, \emph{i.e.} the unital WOT-closed algebras generated by the left and right free shifts or creation operators, are naturally identified with the left and right free multiplier algebras of $F^2 _d$ \cite{Ball2003rkhs,Ball2006Fock}: $$ L^\infty _d \simeq \mr{Mult} ^L (F^2 _d) ; \quad \quad \mbox{and} \quad \quad R^\infty _d \simeq \mr{Mult} ^R (F^2 _d).$$  We will use the notation
$$ \scr{L} _d (\H_1 ,\H _2) := [ \mr{Mult} ^L ( F^2 _d \otimes \H _1 , F^2 _d \otimes \H _2 ) ] _1 = [ L^\infty _d \otimes \L (\H _1 , \H _2 )] _1, $$ and $$
\scr{R} _d (\H _1 , \H_2 ) := [ \mr{Mult} ^R ( F^2 _d \otimes \H _1 , F^2 _d \otimes \H _2 ) ] _1, $$ for the left and right (operator-valued) \emph{free Schur classes}, the closed unit balls of the left and right multipliers between vector-valued Fock spaces over $\C ^d$. Since the left and right free Toeplitz algebras $L^\infty _d$ and $R^\infty _d$ are each others commutants, the space of left multipliers $\mr{Mult} ^L (F^2 _d \otimes \H_1 , F^2 _d \otimes \H _2)$ can also be identified as the spaces of bounded linear maps which intertwine the scalar right multiplier algebras $R^\infty _d \otimes I_{\H _1}$ and $R^\infty _d \otimes I_{\H _2}$ acting on vector-valued Fock spaces. In the case where $\H _1 = \H _2 = \H$, we simply write $\scr{L} _d (\H)$ for $\scr{L} _d (\H , \H ) := [ L^\infty _d \otimes \L (\H) ]_1$.

As in the commutative setting, any element $B = B^L \in \scr{L} _d (\H )$ or $B = B^R \in \scr{R} _d (\H )$ can be used to define a positive \emph{free deBranges-Rovnyak kernel} $\hat{k} ^B$ and corresponding left or right \emph{free deBranges-Rovnyak space} $\scr{H} ^L (B )$ or $\scr{H} ^R (B)$:

\begin{eg}{ Free deBranges-Rovnyak spaces} \label{FreedBReg}

Consider vector-valued Fock space $F^2 _d \otimes \H$.
As in the commutative setting, any formal operator-valued power series $B (Z) \in \L (\H , \J ) \{ Z \}$ is the the left or right free Schur class if and only if
$$ \hat{k} ^L (Z,W ) := \hat{k} (Z,W) - B(Z) \hat{k} (Z,W) B(W) ^* \in \L (\H) \{Z, W^*\}, $$ or
$$ \hat{k} ^R (Z,W) := \hat{k} (Z,W) - B(Z) \bullet_R \hat{k} (Z,W) \bullet_R B(W) ^* \in \L (\H ) \{Z,W^*\}, $$ are free positive kernel functions, respectively, where $\hat{k}$ is the free Szeg\"{o} kernel
of $F^2 _d \otimes \H$ \cite[Theorem 3.1]{Ball2006Fock}.

The (left or right) free deBranges-Rovnyak space is then defined as $\scr{H} ^L (B) := \mc{F} ( \hat{k} ^L )$ or $\scr{H} ^R (B) := \mc{F} (\hat{k} ^R )$, depending on whether $B$ is in the left or right free operator-valued Schur class.

As in the commutative case, $\scr{H} ^R (B)$ can be defined as a complementary range space \cite{Sarason-dB}:
$$ \scr{H} ^R (B)  := \scr{M} \left( \sqrt{I _{F^2 _d \otimes \J} - M ^R _B  (M^R _B)  ^*} \right). $$ Namely, $\scr{H} ^R (B) = \ran{ \sqrt{ I - M_B ^R  (M^R _B ) ^* } }$ equipped with the inner
product that makes $ \sqrt{ I - M_B ^R (M^R _B ) ^* }$ a co-isometry onto its range: if $P$ is the orthogonal projection onto $\ker{ \sqrt{ I - M_B ^R (M^R _B ) ^* }} ^\perp$,
$$ \ip{  \sqrt{ I - M_B ^R (M^R _B ) ^* } h}{ \sqrt{ I - M_B ^R (M^R _B ) ^* } g} _B := \ip{Ph}{g} _{F^2}. $$
In the above, $M_B ^R \in \L (F^2 _d \otimes \H , F^2 _d \otimes \J)$ is
defined by right free multiplication by $B(Z)$ (assume $B$ belongs to the right Schur class). A similar statement, of course, holds if $B$ is in the left Schur class.

To see that $\scr{H}:= \scr{M} \left( \sqrt{I - M_B ^R (M_B ^R) ^* } \right)$ and $\scr{H} ^R (B)= \mc{F} (\hat{k} ^R )$ are the same space, first note that by free RKHS theory, $\scr{H} ^R (B)$ is contractively contained in
$F^2 _d \otimes \J$ since $\hat{k} \otimes I_\J  - \hat{k} ^R$ is a positive free kernel. As in \cite[Section I-3]{Sarason-dB}, $\scr{H}$ is also contractively contained in $F^2 _d \otimes \J$, and if $\hat{k}$ denotes the free (operator-valued) Szeg\"{o} kernel and $f = \sqrt{I - M^R _B  (M^R _B) ^*} g \in \scr{H}$, then
\ba \ip{h}{ f(Z)} _\H & =& \ip{\hat{k} _Z h }{ f} _{F^2} \nn \\
& = & \ip{ \sqrt{I - M^R _B (M ^R _B) ^* } \hat{k} _Z h}{g} _{F^2} \nn \\
& = & \ip{(I - M^R _B (M ^R _B) ^*) \hat{k} _Z h}{f} _{\scr{H}}. \nn \ea
This shows that $\scr{H}$ is a free RKHS with point evaluation maps
$$ K_Z := (I - M^R _B (M ^R _B) ^*) \hat{k} _Z = \hat{k} _Z - M^R _B \hat{k} _Z \bullet_R B(Z) ^*,$$ and free kernel
$$ K (Z,W) := \hat{k} (Z,W) - B(Z) \bullet_R \hat{k} (Z,W) \bullet_R B(W) ^* = \hat{k} ^R (Z, W). $$ This proves that $\scr{H} = \scr{H} ^R (B)$. Note that in the above $\hat{k} _Z \in \L ( \J , F^2 _d \otimes \J ) \{ Z^* \}$ is a formal power series with coefficients in $\L (\J , F^2 _d \otimes \J)$, and we define
the action of $M^R _B, (M^R _B) ^*$ on such formal power series (as well as the above inner products of formal power series) by linearity.
Alternatively, instead of formal manipulations with free formal power series, one can arrive at the same conclusions by repeating the above arguments with the coefficient maps.
\end{eg}

\section{Relationship to Non-commutative function theory}

Free non-commutative function theory provides an alternative and equivalent mathematical framework for defining non-commutative deBranges-Rovnyak spaces associated to the left and right free Schur classes.
In particular, there is a bijection between free RKHS $\mc{F} (K)$ with free kernels $K$, and functional non-commutative (NC) RKHS of free non-commutative (NC) functions defined on NC sets \cite[Theorem 3.20]{Ball-NC}. In this section we briefly describe the relationship between these two theories as they pertain to our program. Our presentation will follow \cite{KVV,Ball-NC}.

One inspiration for free non-commutative function theory is Popescu's free functional calculus for row contractions (and Popescu's theory of free holomorphic functions) \cite{Pop-funcalc,Pop-freeholo,Pop-freeholo2}. Recall that $\A  := \A ^L _d$ denotes the left free disk algebra, the unital operator algebra generated by the left free shift (the row isometery of left creation operators) on the full Fock space, $F^2 _d$ over $\C ^d$. Further recall that the free left multiplier algebra of $F^2 _d = \mc{F} (\hat{k} )$ is $L _d ^\infty$, the unital WOT-closed operator algebra generated by the left free shift, also called the left free Toeplitz algebra. Similarly we define operator-valued extensions of these algebras: given an auxiliary coefficient Hilbert space $\H$, we will abuse notation slightly and write $\A ^L _d \otimes \L (\H )$ and $L ^\infty _d \otimes \L (\H )$ for the operator-valued left free disk algebra, and the left free Toeplitz algebra, respectively. To be precise, we write $\A ^L _d \otimes \L (\H ), L^\infty _d \otimes \L (\H )$ in place of the norm and WOT-closure of these algebraic tensor products. These algebras are the norm, and WOT-closure, respectively, of the unital operator algebras generated by the operator-valued left free shift $L \otimes I_\H$ acting on vector-valued Fock space $F^2 _d \otimes \H$.

The operator algebras $L^\infty _d, R^\infty _d$ are unitarily equivalent via the transposition unitary $U _T : F^2 _d \rightarrow F^2 _d$: Given an orthonormal basis $\{ e_k \}$ of $\C ^d$ and corresponding left and right creation operators $L_k v = e_k \otimes v$, $R_k v = v \otimes e_k$ on $F^2 _d$, a canonical orthonormal basis for $F^2 _d$ is the set $\{ e_\alpha := L^\alpha 1 \} _{\alpha \in \F ^d}$. The unitary $U_T$ is then defined by transposition of the index:
$$ U_T L^\alpha 1 = U_T e_\alpha = e_{\alpha ^T}, $$ and $\alpha ^T$ denotes the transpose of $\alpha \in \F ^d$ defined previously: if $\alpha = i_1 \cdots i_k$, $\alpha ^T = i_k \cdots i_1$. It is easy to check that
$$ U_T L^\alpha = R^{\alpha ^T} U_T, $$ and it follows that $L^\infty _d \simeq R^\infty _d$ are unitarily equivalent. For this reason, when it is not necessary to distinguish between left and right, we will identify $L^\infty _d$ with $R^\infty _d$, and simply use $F^\infty _d$ to denote the \emph{free Toeplitz algebra}. Any $F \in F^\infty _d$ can then be identified with a (unitarily equivalent) transpose-conjugate pair $F = (F^L , F^R) \in L^\infty _d \times R^\infty _d$. In terms of formal power series, if
\be F^L (Z) = \sum _{\alpha} Z^\alpha F_\alpha, \label{freePS} \ee then
$$ F^R (Z) = (F^L (Z) ) ^T = \sum _{\alpha} Z^{\alpha ^T} F_\alpha  = \sum _{\alpha} Z^\alpha F_{\alpha ^T}. $$  This defines a transpose map on free formal power series, $F^R = T \circ F^L$.

Any $F \in L^\infty _d \otimes \L (\H )$ has the `free Fourier series' of equation (\ref{freePS})  which is defined by computing \cite{DPinv}:
$$ F(Z) = \sum _{\alpha \in \F ^d} Z^\alpha F_\alpha := \sum _{\alpha \in \F ^d } (L^\alpha 1) F_\alpha := M_F ^L (1 \otimes I_\H) ; \quad \quad F_\alpha \in \L (\H). $$

Given any $0\leq r < 1$, and any $F \in L^\infty _d \otimes \L (\H )$, one can check as in \emph{e.g.} \cite[Lemma 3.5.2, Theorem 3.5.5]{Sha2013}, that the power series
$$ \sum _{\alpha \in \F ^d } (rL) ^\alpha \otimes F_\alpha, $$ converges in operator norm for $F^2 _d \otimes \H$. This shows that
$$ F _r (Z) := \sum Z ^\alpha r^\alpha F_\alpha \in \mc{A} _d ^L \otimes \L (\H ), $$ belongs to the (operator-valued) free left disk algebra and one can check as in \cite[Proposition 4.2]{Pop-funcalc} that $M_{F_r} ^L $ converges to $M^L _F$ in the strong operator topology as $r \rightarrow 1 ^-$.

It is important to note, however, that as in the case of Fourier series for the classical disk algebra \cite{Hoff}, the partial sums of the free Fourier series for $F \in \mc{A} _d ^L$ may
not converge, even in the strong or weak operator topologies \cite{DPinv}. Instead, any $F \in L^\infty _d$ (or more generally $L^\infty _d \otimes \L (\H )$) can be recovered from its free Fourier series by taking Ces\`{a}ro sums.
Namely, given any $F \in L ^\infty _d \otimes \L (\H)$, the $N$th Ces\`{a}ro sum of $F$, $\Sigma _N (F) \in L^\infty _d \otimes \L (\H )$ is the average of the first $N$ partial sums of the free Fourier series of $F$. As shown in \cite{DPinv}, for any $N \in \N \cup \{ 0 \}$, $\Sigma _N: L^\infty _d \otimes \L (\H) \rightarrow L^\infty _d \otimes \L (\H)$ defines a completely contractive unital map (into free polynomials) so that
$\Sigma _N (F)$ converges in the strong operator topology of $\L (F^2 _d \otimes \H )$ to $F$.

Results of Popescu \cite{Pop-vN,Pop-funcalc,Pop-freeholo} show that any $F \in L^\infty _d \otimes \L (\H )$ can be used to define a function on strict row contractions: If $T \in \left( \L (\J \otimes \C ^d , \J ) \right) _1$ then by
the Popescu-von Neumann inequality
$$ \| p (T _1 , ... ,T_d ) \| _{\L (\J )} \leq \| p (L_1 , ... , L _d ) \| = \| M _{p} ^L \| _{\L (F^2 _d )} ; \quad \quad p \in \L (\H) \langle L \rangle,$$ where $\L (\H) \langle L \rangle = \L (\H) \langle L_1 , ... , L_d \rangle$ denotes the algebra of polynomials in the $d$ free (non-commuting) variables $L_k$ with coefficients in $\L (\H)$. This inequality (and its matrix-valued version) shows that
$$ p (L)  \in \L (\H ) \langle L \rangle  \mapsto p(T_1, ... , T_d ),$$ defines a unital completely contractive algebra homomorphism which can be extended by continuity to $L ^\infty _d \otimes \L (\H )$.

This functional calculus is one of the inspirations for free non-commutative function theory \cite{KVV,Pop-freeholo,Pop-freeholo2}. Here is a brief introduction which is sufficiently general for our purposes: Let $V = \C ^d$, a complex vector space, and consider the disjoint union
$$ V _{nc} := \coprod _{n=1} ^\infty V ^{n\times n}, \quad \quad   V^{n\times n} := V \otimes \C ^{n\times n} = \C ^{n\times n} \otimes \C ^d.$$ Elements $Z \in V^{n\times n}$ are viewed as bounded row operators on $\C ^n$: $Z =(Z_1, ... , Z_d ): \C ^n \otimes \C ^d \rightarrow \C ^n$. Consider the non-commutative (NC) open unit ball $\Om \subseteq V_{nc}$,
$$ \Om := \coprod _{n =1 } ^\infty \Om _n; \quad \quad \Om _n := \left( \C ^{n \times n} \otimes \C ^d \right) _1, $$ each $\Om _n$ is the set of all strict row contractions on $\C ^n$. This set $\Om$ is an example of what is called a non-commutative (NC) set \cite{KVV} (it is closed under direct sums, and it is also both left and right admissable in the terminology of \cite{KVV}).

A function $F : \Om \subset V_{nc} \rightarrow \L (\H )  _{nc} = \coprod \L(\H ) ^{n\times n}$ is called a non-commutative or \emph{free function} if it has the two properties:
$$ F : \Om _n \rightarrow \L (\H)  ^{n\times n}; \quad \quad F \mbox{ is \emph{graded}}, $$  and, if  $Z \in \Om _n$, $W \in \Om _m$, and $\alpha \in \C ^{m\times n}$ obey
$$ \alpha Z = W \alpha, $$ then $$ \alpha F (Z) = F(W) \alpha; \quad \quad F \mbox{ \emph{respects intertwinings}}.$$
The free function $F$ is called:
\bi
\item[(i)]  \emph{locally bounded} if for any $Z \in \Om _n$, there is a $\delta _n > 0$ so that $F$ is bounded on the ball of radius $\delta _n$ about $Z \in \Om _n$.
\item[(ii)] \emph{analytic} or holomorphic on $\Om$ if $F$ is locally bounded and \emph{G\^{a}teaux differentiable}: For any $Z \in \Om _n$ and $W \in V^{n\times n}$, the G\^{a}teaux derivative of $F$ at $Z$ in the direction $W$:
$$ \lim _{t \rightarrow 0} \frac{F(Z+tW) - F(Z)}{t} = \left. \frac{d}{dt} F(Z +tW) \right| _{t =0} =: \delta F (Z) (W), \quad \quad \mbox{exists.}$$
\ei
By \cite[Theorems 7.2 and 7.4]{KVV}, any locally bounded free function $F$ is automatically analytic, and analyticity of $F$ also implies that $F$ has a certain power series representation (Taylor-Taylor series) with non-zero radius of convergence about any $Z \in \Om$ (it also implies $F$ is Fr\'{e}chet differentiable), see \cite[Chapter 7]{KVV}. Moreover, the results of \cite{Pop-freeholo,Pop-freeholo2,KVV} show, remarkably, that many classical results from complex analysis and several complex variables have purely algebraic proofs that extend naturally to this setting.

Let $\mr{Hol} (\Om ) \otimes \L (\H )$ denote the algebra of all free holomorphic functions on the non-commutative (NC) ball $\Om$ taking values in $\L (\H) _{nc}$. As in \cite{Pop-freeholo}, we define the (operator-valued) \emph{free Hardy algebra}, as the algebra of all uniformly bounded free holomorphic functions on this NC domain $\Om$ taking values in $\L (\H) _{nc}$:
$$ H^\infty (\Om) := \{ F \in \mr{Hol} (\Om ) |  \ \| F \| _\infty < \infty \}, $$ where the supremum norm of $F$ over the NC unit ball is
$$ \| F \| _{\infty} := \sup _{Z \in \Om } \| F (Z ) \|. $$ By the results of \cite[Chapter 7]{KVV}, any $H \in H^\infty (\Om ) \otimes \L (\H )$ has a power series
representation: $$ H(Z) = \sum _{\alpha \in \F ^d} Z^\alpha H_\alpha := \sum Z^\alpha \otimes H_\alpha; \quad \quad Z \in \Om, \ H_\alpha \in \L (\H ), $$ which converges absolutely for any $Z \in \Om$, and uniformly
on any closed NC ball $\Om _r := \coprod (\Om _r ) _n, \ (\Om _r ) _n := \left[ \C ^{n\times n} \otimes \C ^d \right] _r$ of radius $0<r<1$ \cite[Theorems 7.10 and 7.2]{KVV}. The following theorem shows that the free analytic Toeplitz algebra and free Hardy algebra are naturally isomorphic and can be viewed as the same object:

\begin{thm}{ (\cite[Theorem 3.1]{Pop-freeholo}, \cite{KVV})}
    The map $\Phi : F^\infty _d \otimes \L (\H ) \rightarrow H^\infty (\Om ) \otimes \L (\H )$ defined by
$$ H(L) := \sum L^\alpha \otimes H_\alpha \in \L (\H) \{ L \} \mapsto H(Z) := \sum Z^\alpha H_\alpha \in \L (\H ) \{ Z \}, $$ is a unital completely isometric isomorphism.
\end{thm}
    Recall that the above power series for $H(L)$ is to be understood as the SOT-limit of Ces\`{a}ro sums.
\begin{remark}
Using the free functional calculus of Popescu, it is not difficult to verify that $\Phi$ is injective, unital, and completely isometric. Surjectivity follows from approximating any $H \in H^\infty (\Om )\otimes \L (\H)$ by the partial sums of its Taylor-Taylor series expansion about $0_n \in \Om _n$ \cite[Chapter 7]{KVV}. We will call $H^\infty (\Om )$ the several-variable \emph{free Hardy algebra}, and under the above identification we will use the terms free Hardy algebra and free Toeplitz algebra interchangeably.
\end{remark}

\begin{remark} \label{NCRKHSmark}
    In recent research, the theory of positive kernel functions and RKHS has also been extended to the free function theory setting \cite{Ball-NC}. In particular, it can be shown that the class of all free formal RKHS is naturally isomorphic to the class of non-commutative reproducing kernel Hilbert spaces (NC-RKHS) \cite[Theorem 3.20]{Ball-NC}. A NC-RKHS can be viewed as a sort of reproducing kernel Hilbert space of free or non-commutative functions on a NC set.  In particular, one can naturally identify or view our free deBranges-Rovnyak spaces as NC-RKHS of this type.  We have found, however, that the free extension of our commutative Aleksandrov-Clark theory from \cite{Jur2014AC,JM}, seems to carry over most naturally using the formalism of free RKHS. Namely, many of the theorems and proofs of this paper are formally identical (or very similar) to those of \cite{JM}, upon replacing formal point evaluation maps $K _Z$ with the point evaluation maps $K _z$, $z \in \B ^d$.
\end{remark}

\section{Free Herglotz functions and Aleksandrov-Clark maps}
\label{Herglotzsect}

In this section we define free Herglotz functions and construct the free Aleksandrov-Clark maps associated to any element of the free operator-valued Schur classes. Our calculations here are a formal analogue of the approach in \cite{JM} for the commutative Schur class of Drury-Arveson space. As in the previous section, consider the NC set $\Om = \coprod \Om _n$, where $\Om _n = \left( \C ^{n\times n} \otimes \C ^d \right) _1$ is the set of all strict row contractions on $\C ^n$. In what follows we initially focus on the left case, analogous results hold for the right case.

\begin{defn} The \emph{free left Herglotz-Schur class}, $\scr{L} _d ^+ (\H)$, is the set of all free holomorphic $\L (H) _{nc}$-valued functions $H^L (Z) \in \L (\H) \{ Z \}$ on the NC unit ball $\Om$ such that the \emph{left free Herglotz kernel}:
$$ \hat{K} ^{L} (Z,W) := \frac{1}{2} \left( H ^L (Z) \hat{k} (Z,W) + \hat{k} (Z,W) H^L (W) ^* \right) \in \L(\H) \{ Z, W^* \}, $$ is a positive formal free kernel.
\end{defn}
This expression for $\hat{K} ^L$ converges in operator norm for fixed $Z,W \in \Om _n$, and this implies, in particular, that $\re{H^L (Z) } \geq 0$ for all $Z \in \Om _n$ \cite{Ball-NC}. That is, $H^L (Z)$ is a bounded, accretive operator for any $Z \in \Om$. It then follows as in
\cite[Chapter IV.4]{NF}, that $H ^L (Z) +I$ is invertible, and that
$$ B _H ^L (Z) := (H^L (Z) + I ) ^{-1} (H^L (Z) -I ) \in  \L (\H ) \{ Z \}$$ is contraction-valued on the NC unit ball $\Om$ so that $B^L _H \in  [ H^\infty (\Om ) \otimes \L (\H) ] _1 = \scr{L} _d (\H)$ belongs to the free left Schur class. Moreover, $I -B^L _H (Z) = 2 (H^L (Z) +I)$ is invertible for any $Z \in \Om$, and the free deBranges-Rovnyak kernel $\hat{k} ^L$ of $B^L _H$ is given by $$ \hat{k} ^{L} (z,w) = \left( I -B^L _H (Z) \right) \hat{K} ^L (Z,W) \left( I -B^L _H (W) ^* \right). $$ The free right Herglotz-Schur class, $\scr{R} _d ^+ (\H )$, is defined similarly, and given $H^L \in \scr{L} _d ^+ (\H )$, it easy to see that the formal transpose maps $\scr{L} _d ^+ (\H)$ onto $\scr{R} _d ^+ (\H)$ and if we define $H^R := T \circ H^L$, then $B ^R _H = T \circ B^L _H$.

Conversely, let $B = (B^L , B^R ) \in \scr{L} _d (\H) \times \scr{R} _d (\H) $ be a free Schur class transpose-conjugate pair. Motivated by the above, we will assume that any such $B^L , B^R$ are \emph{non-unital} in the sense that $I - B^L (Z), \ I - B^R (Z)$ are invertible for any fixed $Z \in \Om$. Given such a pair, $B$, one can define a transpose-conjugate pair of free holomorphic functions $H _B = (H^L _B , H^R _B)$ on $\Om$ by
$$ H^L _B (Z) := (I - B^L (Z) ) ^{-1} (I + B^L (Z)); \quad \quad Z \in \Om, $$ and similarly for $H^R _B$. The free Herglotz kernel for $H^L _B$ is then
\ba \hat{K} ^L (Z,W) & = & \frac{1}{2} \left( H^L _B (Z) \hat{k} (Z, W) + \hat{k} (Z,W) H^L _B (W) ^* \right) \nn \\
& = & (I - B^L (Z) ) ^{-1} \hat{k} ^L (Z,W) (I - B^L (W) ^* ) ^{-1}, \nn \ea where $\hat{k} ^L$ is the free left deBranges-Rovnyak kernel for $\scr{H} ^L (B)$. It follows that $\hat{K} ^L$ (and similarly $\hat{K} ^R$) are positive free kernels so that $H _B = (H^L _B , H^R _B)$ is a transpose-conjugate
pair of free Herglotz-Schur functions on $\Om$. It is easy to verify that the maps $B \mapsto H_B$ and $H \mapsto B_H$ are compositional inverses and define bijections between the non-unital free Schur classes and the free Herglotz-Schur classes.

\begin{remark}
    The assumption that a free Schur pair $B = (B^L , B^R) \in \scr{L} _d (\H ) \times \scr{R} _d (\H )$ be non-unital is not very restrictive. A simple argument combining the free Schwarz lemma for free holomorphic functions on the NC unit ball $\Om$ (see \cite[Theorem 2.4]{Pop-freeholo})
with automorphisms of the unit ball of $\L (\H )$ shows that $B(Z)$ is strictly contractive on the NC unit ball $\Om$ if and only if $B(0) = B _\emptyset$ is a strict contraction (for $0 \in \Om _n$), and this happens if and only if $b(0) = B^L _\emptyset = B^R _\emptyset$ is a strict contraction, where $b \in \scr{S} _d (\H )$ is the image of $B ^L$ or $B^R$ under the symmetrization (quotient by the commutator ideal) map. We say $B$ is \emph{strictly contractive} if this holds, and certainly any strictly contractive $B$ is non-unital.

It seems reasonable that the assumption that $B$ be non-unital can be relaxed if one is willing to allow $H^L , H^R$ to take values in unbounded operators see \cite[Remark 1.10]{JM}. We will avoid such complications and assume throughout that $B$ is non-unital.
\end{remark}

 Given any non-unital $B = (B^L , B^R ) \in \scr{L} _d (\H ) \times \scr{R} _d (\H )$, we define the left free Herglotz space,
$\scr{H} ^{L, +} (H_B ):= \mc{F} (\hat{K} ^L )$, as the free RKHS corresponding to the free left Herglotz kernel $\hat{K} ^L$ of $H_B ^L$.  The above relationship between the left free deBranges-Rovnyak and left free Herglotz kernels shows that there is a natural unitary multiplier from $\scr{H} ^L (B)$ onto $\scr{H} ^{L, +} ( H_B)$:

\begin{lemma} \label{canformmult}
Given any non-unital $B \in \scr{L} _d (\H) $, formal left multiplication by $I - B(Z)$ is an isometry, $M^L _{(I-B)}$, of the left free Herglotz space $\scr{H} ^{L, +} (H_B)$ onto the left free deBranges-Rovnyak space $\scr{H} ^L (B)$. The action of this isometry on formal point evaluation maps is:
$$ M ^L _{(I-B)} \hat{K} _W ^L = (M ^L _{(I-B) ^{-1}} ) ^* \hat{K} _W ^L = \hat{k} _w ^L (I -B (W) ^* ) ^{-1} \in \L (\H , \scr{H} ^L (B) ) \{ W^* \}. $$
\end{lemma}

Given any fixed left free Herglotz function $H^L$, define a map $\phi : \mc{A} + \mc{A} ^* \rightarrow \L (\H )$ by
$$ \phi (I ) := \re{ H _\emptyset } \geq 0 ; \quad \quad \phi (L^{\alpha ^T} ) ^* := \frac{1}{2} H_\alpha; \ \alpha \neq \emptyset, $$ where the $H_\alpha \in \L (\H )$ are the coefficients of the formal power series for $H^L $.  Extend $\phi$ so that it is self-adjoint and linear.
It follows that
$$ H ^L (Z) = 2 \sum _\alpha Z^\alpha \phi (L^{\alpha ^T}) ^* - \phi (I), $$ by definition. Let $CP (\A ; \H)$ denote the set of all completely positive maps from $\A + \A ^*$ into $\L (\H )$ (we simply write $\A +\A ^*$ in place of its norm closure). Recall here that $\mc{A} := \mc{A} _d ^L$ is the left free disk algebra.

\begin{prop} \label{Herglotzkern}
    The free left Herglotz kernel of $H^L$, $\hat{K} ^L (Z,W)$, has the form
    $$ \hat{K} ^L (Z,W) = \sum _{\alpha , \beta } Z^\alpha (W^*)^{\beta ^T} \phi ( (L ^{\alpha ^T})  ^* L ^{\beta ^T} ), $$ and the map $\phi$ belongs to $CP (\mc{A} ; \H )$.
\end{prop}

It will be useful to first show that any positive element in $\mc{A} + \mc{A}^*$ is the limit of `sums of squares':
Let $\mc{C} := [ \A + \A^* ] _+$, the positive norm-closed cone of the (norm-closed) operator system $\A +\A ^*$, and let $\mc{C} _0 := [ \mc{A} ^* \mc{A} ] _+$, \emph{i.e.} $\mc{C} _0$ is the positive norm-closed cone of elements which are `sums of squares':
$$ p \in \mc{C} _0 \quad \Rightarrow \quad p = \sum a_k ^* a_k ; \quad \quad a_k \in \mc{A}. $$
\begin{lemma}
Any positive element of $\A + \A ^*$ is the norm-limit of sums of squares, \emph{i.e.}, $\mc{C} _0 = \mc{C}$.
\end{lemma}

\begin{proof}
 Suppose not. Then there is a positive $p \geq 0$ in $\mc{A} + \mc{A} ^*$ so that $p \in \mc{C} \sm \mc{C} _0$. By the Minkowski cone separation theorem,
there is a real linear functional $\la : \mc{C} \rightarrow \R$ so that $\la (q) \geq 0$ for all $q \in \mc{C} _0$ but $\la (p) < 0$.

We can extend $\la$ to a bounded complex linear functional on $\mc{A} + \mc{A} ^*$ in the usual way: If $x$ is self-adjoint in $\mc{A} + \mc{A}^*$ then $x = p-q$ for $p, q \in \mc{C}$. Then let $\La (x) := \la (p) - \la (q)$, and if $x = r +is$ in $\mc{A} + \mc{A} ^*$ with $r,s$ self-adjoint in $\mc{A} + \mc{A}^*$ then define
$\La (x) := \La (r) +i \La (s)$. This is possible since $\mc{A} + \mc{A} ^*$ is a unital operator system so that any self-adjoint element in $\mc{A} + \mc{A}^*$ can be written as the difference of elements of $\mc{C}$ (and the real and imaginary parts of any $x \in \mc{A} +\mc{A} ^*$ are also in the operator system). We will simply write $\la$ in place of its extension $\La$ to $\mc{A} + \mc{A} ^*$.

Define a quadratic form on $\mc{A}$ by:
$$ \ip{a}{b} _\la := \la (a^* b) \in \C; \quad \quad a,b \in \mc{A}. $$ This is a positive quadratic form or pre-inner product on $\mc{A}$,
$$ \ip{a}{a} _\la = \la (a^* a ) \geq 0 ; \quad \quad a \in \mc{A}, $$ since $a^* a \in \mc{C} _0$. As in the usual Gelfand-Naimark-Segal (GNS) construction if $N_\la \subset \mc{A}$ is the closed subspace of vectors of length zero with respect to $\ip{\cdot}{\cdot} _\la$, then this pre-inner product promotes to an inner product on $$ \frac{\mc{A}}{N_\la}, $$ and we let $\mc{H} _\la$ denote the Hilbert space completion of this inner product space.

We can also define a GNS representation $\pi _\la : \mc{A} \rightarrow \L (\H _\la )$ in the usual way:
$$ \pi _\la (a) (b + N_\la ) := ab + N _\la. $$ This is well-defined since $N_\la$ is a closed left $\mc{A}$-module. It is not hard to see that
$\pi _\la$ is a completely contractive and unital representation of $\mc{A}$, and so it extends naturally to a completely positive unital map on $\mc{A} + \mc{A} ^*$. Since $p \geq 0$ and $\pi _\la$ is positive, it follows that $\pi _\la (p) \in \L (\H _\la )$ is a positive operator. This produces the contradiction:
$$ \ip{1}{\pi _\la (p) 1} _\la = \la (p) < 0, $$ and we conclude that $\mc{C} _0 = \mc{C}$.
\end{proof}

\begin{proof}{ (of Proposition \ref{Herglotzkern} )}
    Let $\hat{K} := \hat{K} ^L$. We have that
\ba  2 \hat{K} (Z,W) & = & H ^L (Z) \hat{k} (Z,W) + \hat{k} (Z,W) H ^L (W) ^* \nn \\
& = & \sum _{\alpha, \beta} Z^{\alpha \beta} (W^* ) ^{\beta ^T} H_\alpha + \sum _{\alpha , \beta } Z^\alpha (W^*) ^{\alpha ^T \beta ^T} H_\beta ^* \nn \\
& = & \sum _{\ga , \beta } Z^\ga (W^* ) ^{\beta ^T} \left( \sum _{\alpha \beta = \ga } H_\alpha \right) + \sum _{\alpha , \ga } Z^\alpha (W^*) ^{\ga ^T} \left( \sum _{\beta \alpha = \ga } H_\beta ^* \right) \nn \\
& = & \sum _{\alpha, \beta } (Z) ^\alpha (W^*) ^{\beta ^T} \left( \sum _{\ga \beta = \alpha } H_\ga + \sum _{\ga \alpha = \beta } H_\ga ^* \right), \nn \ea
and this calculation shows that the coefficient kernel of the free positive kernel $\hat{K} $ is:
$$ \hat{K} _ {\alpha , \beta} := \frac{1}{2} \left( \sum _{\ga \beta = \alpha } H_\ga + \sum _{\ga \alpha = \beta } H_\ga ^* \right). $$
In particular it follows that
$$ \hat{K}_{\alpha , \emptyset} = \frac{1}{2} H_\alpha = \phi (L ^{\alpha ^T} ) ^*, $$ by definition.

Now suppose that $\alpha = \la  \beta $ and observe that
\ba 2 \hat{K}_{\la \cdot \beta , \beta } & = & \sum _{\ga \cdot \beta = \la \cdot \beta} H_\ga + \sum _{\ga \cdot \la \cdot \beta = \beta} H_\ga ^* \nn \\
& = & H_\la \nn \\
& = & 2 \hat{K} _{\la , \emptyset}. \nn \ea It follows that if $\alpha = \la \beta$, the map $\phi $ obeys
$$ \phi ( (L ^{\alpha ^T} ) ^*  L^{\beta ^T} ) = \phi ( L ^{\la ^T} ) ^*, $$ so that $\phi$ is well-defined on $\A +\A ^*$. In order to arrive at the above equation, observe that it was necessary that the transpose appears in the definition $2 \phi (L^* ) ^\alpha = 2 \phi (L ^{\alpha ^T} ) ^* = H_\alpha$. Since, for fixed $\alpha, \beta \in \F ^d$,
 \ba (L ^{\alpha ^T} ) ^* L^{\beta ^T} & = & \left\{ \begin{array}{cc} (L^{\ga ^T} ) ^*; & \ga \cdot \beta = \alpha \\ L^{\ga ^T}; & \ga \cdot \alpha =\beta \\ 0 & \mbox{else} \end{array} \right. \nn \\
& = & \sum _ {\ga ; \ \ga \cdot \beta = \alpha } (L^{\ga ^T})  ^*  + \sum _{\ga ; \ \ga \cdot \alpha = \beta} L^{\ga ^T}, \nn \ea  it follows that
$$ \hat{K} _{\alpha , \beta} = \phi \left( (L^{\alpha ^T}) ^* L^{\beta ^T} \right). $$

The fact that $ \hat{K}_{\alpha ,\beta }$ is a positive free coefficient kernel will imply that $\phi $ is completely positive: Indeed, consider any element $A \in \mc{A} \otimes \C ^{n \times n}$
of the form $$ A = \sum _{k=1} ^N L^{\alpha _k} \otimes C_k; \quad \quad \alpha _k \in \F ^d, \ C_k \in \C ^{n\times n}. $$ The set of all such finite sums is norm dense in $\mc{A} \otimes \C ^{n\times n}$. To show that $\phi$ is completely positive, the (matrix-version of the) previous sums of squares lemma implies that it is sufficient to show that
$$ \phi ^{(n)} (A ^* A )  =  (\phi \otimes \mr{id} _n ) \left( \sum _{k,j} (L ^{\alpha _k}) ^* L ^{\alpha _j} \otimes C _k ^* C_j  \right) \geq 0, $$ for all $n \in \N$. The above can be written as
\ba  \phi ^{(n)} (A ^* A ) & = & \sum _{k,j =1} ^N \phi \left( (L^{\alpha _k}) ^* L^{\alpha _j } \right) \otimes C_k ^* C_j \nn \\
& = & \sum _{k,j =1} ^N \hat{K}_{\alpha _k ^T , \alpha _j ^T } \otimes C_k ^* C_j \nn \\
& = & \left( \sum _{k =1} ^N \hat{K} _{\alpha _k ^T} \otimes C_k \right) ^* \left(   \sum _{j =1} ^N \hat{K} _{\alpha _j ^T} \otimes C_j \right) \nn \\
& \geq & 0, \nn \ea and this proves that $\phi$ is completely positive.
\end{proof}
Consider the \emph{free Cauchy kernel}
 \ba (I -ZL ^* ) ^{-1} & := & \sum _{k=0} ^\infty (ZL ^* ) ^k = \sum _{\alpha \in \F ^d } Z^\alpha (L^*) ^\alpha \in L^\infty _d \{ Z \} \nn \\
  & = & \sum _\alpha Z^\alpha (L^{\alpha ^T} ) ^*. \label{FreeCK} \ea
With this definition it follows that
\ba \hat{K} ^L (Z,W) & = &  \phi \left(  (I-ZL^*) ^{-1}  * \circ (I - WL ^* )^{-1}  \right) \nn \\
& =: & \phi \left(  (\sum Z^{\alpha } (L^{\alpha ^T}) ^* )  ( \sum _\beta W^\beta (L^*) ^\beta ) ^* \right) \nn \\
& = &  \phi \left(  (\sum Z^\alpha (L^{\alpha ^T} ) ^* )  \sum _\beta (W^*) ^{\beta ^T} L ^{\beta ^T} \right) \nn \\
& = & \sum Z^\alpha (W^*) ^{\beta ^T} \phi \left ( (L ^{\alpha ^T} ) ^* L ^{\beta ^T} \right). \nn \ea
In the above, $*$ denotes the formal adjoint defined previously.

With these definitions we also have that
$$ H  ^L (Z) =  \phi \left( (I -ZL^* ) ^{-1} (I +ZL ^* ) \right) + i \im{H_\emptyset }, $$ or equivalently,
$$ H ^L (Z) =  \phi \left( 2 (I -ZL^* ) ^{-1} - I \right) + i \im{ H_\emptyset } .$$ This is the left free Herglotz formula, and it is clearly
a non-commutative formal analogue of the classical Herglotz formula (\ref{classicH}) from the introduction, as well as a direct free analogue of the commutative results for $\scr{S} _d (\H)$ obtained in \cite{Jur2014AC,JM}.

This argument is reversible. Given any $\phi \in CP (\mc{A} ; \H )$ define a positive free kernel $\hat{K} = \hat{K} ^L$ and coefficient kernel $\hat{K} _{\alpha, \beta} $ by
$$ \hat{K} ^L (Z,W) := \sum  Z^\alpha (W^*) ^{\beta ^T} \hat{K} _{\alpha , \beta} = \phi \left( (I-ZL^* ) ^{-1} * \circ (I -WL^* ) ^{-1}  \right), $$ and
\be \label{freeHcoker} \hat{K} _{\alpha , \beta } := \phi ( (L^{\alpha ^T}   ) ^* L^{\beta ^T} ). \ee
 Complete
positivity of $\phi$ ensures that this defines a positive coefficient kernel.
If one defines
$$H ^L (Z) := \phi \left( (I-ZL ^* ) ^{-1} (I + Z L ^* ) \right), $$ it follows that $H^L$ is a free holomorphic $\L (\H )$-valued function on the NC unit ball $\Om$, and one can calculate that
\be \hat{K} ^L (Z,W) = \frac{1}{2} \left( H ^L (Z) \hat{k} (Z,W) + \hat{k} (Z,W) H^L (W) ^* \right). \label{HergRK} \ee
Indeed,
\ba \hat{K} ^L (Z,W) & = & \sum _{\alpha , \beta} Z^\alpha (W^* ) ^{\beta ^T} \phi ((L ^{\alpha ^T})^* L ^{\beta ^T} ) \nn \\
& = & \left( \sum _{\alpha ^T \geq \beta ^T } \sum _\beta + \sum _{\beta ^T \geq \alpha ^T} \sum _\alpha \right)   \phi ((L ^{\alpha ^T} ) ^* L ^{\beta ^T} ) - \phi (I) \hat{k} (Z,W). \nn \ea
Consider the first sum. Since $\alpha ^T \geq \beta ^T$ it follows that $\alpha ^T = \beta ^T \ga ^T$ or $\alpha = \ga \beta$ for some $\ga \in \F ^d$. This first sum can then be written as:
\ba \sum _{\ga } \sum _\beta  Z ^{ \ga \beta } (W^*) ^{\beta ^T} \phi \left( (L^{\beta ^T \ga ^T} ) ^* L^{\beta ^T} \right) & = & \sum _{\ga , \beta}  Z^{\ga \beta } (W^* ) ^{\beta ^T} \phi (L ^{\ga ^T} ) ^* \nn \\
&= & \frac{1}{2} \left( H ^L (Z) + \phi (I) \right) \hat{k} (Z,W). \nn \ea
The full calculation then establishes the formula (\ref{HergRK}). Since this is a positive free kernel, it follows that $H ^L \in \scr{L} _d ^+ (\H )$ belongs
to the left free Herglotz-Schur class.

The entire above analysis can be repeated with right free Herglotz-Schur functions. Given a right free $H ^R \in \scr{R} _d ^+ (\H )$ we can define
$\phi \in CP (\mc{A} ; \H )$ by
$$ \phi ( (L^\alpha ) ^* L^\beta ) := \hat{K} ^R _{\alpha , \beta }. $$ Then,
$$ \hat{K} ^{R} (Z, W) = \phi  \left( T \circ (I-ZL^*) ^{-1}   * \circ T \circ (I - WL^* )^{-1}  \right), $$ where
$$ * \circ T \circ (I - WL^* ) ^{-1} =  * \circ T \circ  \sum W^\beta (L^* ) ^\beta = \sum (W^* ) ^{\beta ^T} L ^\beta, $$ and $T$ is the formal transpose defined previously. Also note that $*\circ T = T \circ *$.

In this right case we obtain the right Herglotz formula
$$ H  ^R (Z) := T \circ \phi \left( (I -ZL^* ) ^{-1} (I +ZL ^* ) \right) +i \im{ H _\emptyset } = T \circ H^L (Z). $$
Conversely, given $\phi \in CP (\mc{A} ; \H)$, one can define the right Herglotz function as above and it follows that any $\phi \in CP (\mc{A}; \H )$ corresponds uniquely to a transpose-conjugate pair of left and right free Herglotz-Schur functions $H = (H ^L , H^R ) \in \scr{L} _d ^+ (\H ) \times \scr{R} _d ^+ (\H )$. These arguments and formulas define bijections (modulo imaginary constant operators) between transpose-conjugate Herglotz-Schur pairs and completely positive maps on the free disk operator system. In summary:

\begin{thm} \label{FreeHform}
There are bijections between the three classes of objects:
\bi
    \item[(i)] Transpose-conjugate pairs $B = (B^L , B^R ) \in \scr{L} _d (\H) \times \scr{R} _d (\H )$ of non-unital free Schur class functions.
    \item[(ii)] Transpose-conjugate pairs $H = (H^L , H^R ) \in \scr{L} ^+ _d (\H) \times \scr{R} _d ^+ (\H )$ of free Herglotz-Schur functions.
    \item[(iii)] The positive cone $CP (\A ; \H )$ of completely positive maps from the free disk operator system $\mc{A} + \mc{A} ^*$, $\mc{A} = \mc{A} _d$, into $\L (\H )$.
\ei
The bijection between free Schur class pairs and free Herglotz-Schur class pairs is given by the maps $B \mapsto H_B$ and $H \mapsto B_H$. The bijection (modulo imaginary constants) between
the free Herglotz-Schur classes and $CP (\A ; \H)$, $H = (H^L , H^R) \in \scr{L} _d ^+ (\H) \times \scr{R} _d ^+ (\H ) \leftrightarrow \phi \in CP (\mc{A} ; \H)$, is given by the free Herglotz formulas:
$$ H ^L _B (Z) := \phi \left( (I -ZL^*) ^{-1} (I + ZL^*) \right) +i \im{H_\emptyset}; \quad \quad \mbox{and} \quad H^R _B (Z) := T \circ H^L _B (Z).$$
\end{thm}

Again, observe that the above formula is formally analogous to the classical Herglotz representation formula (\ref{classicH}) for Herglotz functions on the disk. (It recovers the classical formula in the scalar-valued, single-variable case if we identify AC measures on the unit circle with positive linear functionals on the classical disk algebra.)

\begin{defn}
    We will use the notation $\mu _B \in CP (\A ; \H )$ for the completely positive map which corresponds uniquely to the transpose-conjugate pair $B := (B^L, B^R)$ (equivalently to $H_B = (H_B ^L , H_B ^R)$) by the above theorem. The map $\mu _B$ will be called the \emph{Aleksandrov-Clark map} or non-commutative Aleksandrov-Clark measure of $B$.
\end{defn}

\section{The free Cauchy transforms}

As in \cite{Jur2014AC,JM}, given any $\phi=\mu _B \in CP (\mc{A} ; \H )$ one can construct a Gelfand-Naimark-Segal (GNS)-type space, $F ^2 (\mu _B )$, and associated Stinespring representation $\pi _\phi = \pi _B : \mc{A} + \mc{A} ^* \rightarrow \L (F^2 (\mu _B) )$. Here $B = (B^L , B^R)$ is the unique transpose conjugate pair of free Schur class elements corresponding to $\phi$. This construction relies on the \emph{semi-Dirichlet property} of the free disk algebra $\mc{A}$ \cite{Dav2011}: $$ \mc{A} ^* \mc{A} \subseteq ( \mc{A} + \mc{A} ^* ) ^{-\| \cdot \| }.$$ Briefly, given $\phi = \mu _B$,
consider the algebraic tensor product $ \mc{A} \otimes \H,$ endowed with the pre-inner product
$$ \ip{ a_1 \otimes h_1}{ a_2 \otimes h_2} _B := \ip{h_1}{\mu _B (a_1 ^* a_2) h_2 } _\H. $$ The fact that $\mu _B (a_1 ^* a_2)$, and hence that this pre-inner product is well-defined relies on the semi-Dirichlet property of $\A$. If $N_B$ denotes the closed left $\A$-module (or left ideal in $\A$) of all vectors of length zero in this algebraic tensor product, then $\ip{\cdot}{\cdot} _B$ promotes to an inner product on the quotient space
$$ \frac{\mc{A} \otimes \H}{N_B}, $$ and the Hilbert space completion of this inner product space will be denoted by $F^2 ( \mu _B)$, the \emph{free Hardy space} of $ \mu _B$.
The associated Stinespring representation is defined by $a \mapsto \pi _B (a)$ where
$$ \pi _B (a) (a' \otimes h  + N_B ) := aa' \otimes h + N_B. $$ The representation $\pi _B : \mc{A} \rightarrow \L (F^2 (\mu _B))$ is a unital completely isometric isomorphism which is \emph{$*$-extendible} to a $*$-representation of the Cuntz-Toeplitz $C^*$-algebra $\mc{E} := C^* (\mc{A} )$ (and is well-defined since $N_B$ is a left ideal). In particular it follows that
$\pi _B (L)$ is a row-isometry on $F^2 (\mu _B ) \otimes \C ^d$. This yields the Stinespring dilation formula:
$$ \mu _B ( L^\alpha )  =  [I \otimes ] _B ^* \pi _B (L ) ^\alpha [I \otimes ] _B; \quad \quad \alpha \in \F ^d, $$ where the bounded linear embedding
$[I\otimes ]_B : \H \rightarrow F^2 (\mu _B)$ is defined by
$$ [I\otimes ]_B h := I\otimes h + N_B \in F^2 ( \mu _B), $$ and $\| [I\otimes ] _B \| ^2 = \| \mu _B (I) \| $. This embedding is isometric if and only if $\mu _B$ is unital.

Recall that a CP map $\phi = \mu _B \in CP (\mc{A} ; \H )$ defines both a left and right free Herglotz space with free kernels $\hat{K} ^L , \hat{K} ^R$, respectively.
In what follows we consider the right case. The left case is, as usual, analogous. The formal point evaluation map $\hat{K} ^R _Z$ is given by the free formal series:
$$ \hat{K} _Z  := \sum _\alpha (Z^*) ^{\alpha ^T} \hat{K}  _\alpha ^R. $$ Let $B^R$ be the right Schur class element defined by $\mu _B$.
We define the \emph{free right Cauchy transform}:
$$ \hat{\mc{C}} _R : F^2 (\mu _B ) \rightarrow \scr{H} ^{R,+}  (H_B), $$ by
\be \hat{\mc{C} } _R \left( * \circ T  \circ [I \otimes ] _B ^* (I -Z\pi _B (L) ^* ) ^{-1}  \right) := \hat{K} _Z ^R  \in \L (\H) \{ Z^* \}. \label{coeffequal} \ee
Expanding the above in free formal power series,
$$ \hat{\mc{C} } _R \left( \sum _{\alpha} (Z^* ) ^{\alpha ^T} \pi _B (L ) ^\alpha [I \otimes ] _B \right) = \hat{K} _Z ^R, $$ so that
in terms of coefficient maps,
$$ \hat{\mc{C}} _R (\pi _B (L ) ^\alpha [I\otimes ] _B) = \hat{K} _\alpha ^R. $$
\begin{remark}
Both the left and right hand sides of the above equation (\ref{coeffequal}) are free power series in $Z^*$. To say that they are equal is to say that their coefficients are equal. We then extend the action of $\hat{\mc{C} } _R$ to free power series by linearity.
\end{remark}

The free right Cauchy transform is an onto linear isometry since:
\ba & & \left( *\circ T \circ [I\otimes ] _B ^* (I -Z\pi _B (L) ^* ) ^{-1}  \right) ^* \left( * \circ T \circ [I\otimes _B ] ^* (I - W\pi _B (L) ^*) ^{-1}  \right) \nn \\
& = &  [I\otimes ]_B ^* T \circ (I - Z \pi _B (L) ^* ) ^{-1}  * \circ T \circ [I\otimes ] _B ^* (I - W \pi _B (L) ^* ) ^{-1} \nn \\
& = & [I\otimes ] _B ^* \sum _{\alpha} Z^ \alpha \pi _B (L ^\alpha ) ^*  \sum _\beta (W^*) ^{\beta ^T} \pi _B (L ) ^\beta [ I \otimes ] _B \nn \\
& = & \sum _{\alpha, \beta} Z^\alpha (W^* ) ^{\beta ^T} [I \otimes ] _B ^* \pi _B (L ^\alpha ) ^* \pi _B (L) ^\beta [I \otimes ] _B  \nn \\
& =& \sum _{\alpha , \beta } Z^{\alpha } (W^*) ^{\beta ^T}  \mu _B \left( (L^\alpha )^* L^\beta \right) \nn \\
& = & \hat{K} ^R (Z,W), \nn \ea or, equivalently,
\ba & &  \left( \hat{\mc{C}} _R \pi _B (L) ^\alpha [I \otimes ] _B \right) ^* \left( \hat{\mc{C}} _R \pi _B (L) ^\beta [I \otimes ] _B \right) \nn \\
& = & [I\otimes ] _B ^* \pi _B (L^\alpha ) ^* \pi _B (L ) ^\beta [I \otimes ] _B \nn \\
& = & \hat{K} ^R _{\alpha, \beta } = \mu _B \left( (L^\alpha ) ^* L^\beta \right) \quad \left( = \hat{K} ^L _{\alpha ^T , \beta ^T} \right). \nn \ea

The \emph{weighted free right Cauchy transform} $\hat{\mc{F} } _R  : F^2 (\mu _B ) \rightarrow \mathscr{H} ^R (B)$ is then defined by
$$\hat{\mc{F}} _R := (I - B^R (Z) ) \bullet_R \hat{\mc{C} } _R ,$$ an onto isometry. As in Lemma \ref{canformmult} of Section \ref{Herglotzsect},
free right multiplication by $(I - B^R (Z) )$ is an isometry of the free right Herglotz space $\scr{H} ^{R, +} (B)$ onto the free right deBranges-Rovnyak space
$\scr{H} ^R (B)$, and the inverse or Hilbert space adjoint of this isometry acts as free right multiplication by $(I - B^R (Z) ) ^{-1}$ so that
$M _{(I-B)} ^R = (M_{(I-B) ^{-1} } ^R ) ^*$.
It follows that
\ba \hat{\mc{F}} _R \left( * \circ T \circ [I\otimes ] _B ^* (I -Z \pi _B (L) ^* ) ^{-1} \right) & = & (M^R _{(I-B)^{-1}} ) ^* \hat{K} _Z ^R \nn \\
& = & \hat{k} _z ^R \bullet _R (I - B (Z) ^* ) ^{-1}. \label{wCTformula} \ea

Similarly we can define the left free Cauchy and weighted Cauchy transforms, $\lft, \lct$ by
$$ \hat{\mc{C} } _L \left( * \circ [I\otimes ] _B (I -Z\pi _B (L) ^* ) ^{-1} \right) := \hat{K} ^L _Z, $$ or on coefficient maps as
$$ \lct (\pi _B (L ) ^{\alpha ^T}) ^* [I\otimes ] _B := \hat{K} ^L  _\alpha, $$ and $\lft := (I - M ^L _B ) \circ \lct$.

\begin{prop}
    Let $B := (B^L , B^R) \in \scr{L} _d (\H) \otimes \scr{R} _d (\H)$ be a transpose conjugate pair. The onto isometry $W_T := \lft \rft ^* : \scr{H} ^R (B) \rightarrow \scr{H} ^L (B)$ acts by transposition: If $F (Z) = \sum _\alpha Z^\alpha F_\alpha \in \scr{H} ^R (B)$ then
$(W_T F) (Z) = \sum _\alpha Z^{\alpha ^T} F_\alpha. $
\end{prop}

The proof is easily verified, and omitted.

\section{The Free Clark formulas} \label{Clarksection}

Assume that $\phi =\mu _B \in CP (\mc{A} ; \H )$ where $B = (B^L , B^R) \in \scr{L} _d (\H ) \times \scr{R} _d (\H )$ is a transpose-conjugate pair of free (operator-valued) Schur multipliers. In this section we will develop right free analogues of the Clark unitary perturbation formulas, the left case is analogous. Our approach and proof is a direct free analogue of the proof of the Clark intertwining formulas for the commutative setting of Schur $b \in \scr{S} _d (\H )$. \cite[Theorem 4.16, Section 4]{JM}.

A significant complication appears in the commutative Aleksandrov-Clark theory as soon as $d>1$. Namely, in contrast to the classical single-variable theory \cite{Sarason-dB}, the
deBranges-Rovnyak spaces $\scr{H} (b)$ for $b \in \scr{S} _d (\H )$ are generally not invariant for the adjoints of the components of the Arveson $d$-shift on $H^2 _d$ \cite{Ball2008}. The appropriate replacement for
the restriction of the backward shift in the several-variable theory is a \emph{contractive Gleason solution} for $\scr{H} (b)$ \cite{Ball2008,Ball2010,Ball2011dBR,Glea1964,Alpay2002}. Here, (see \emph{e.g.} \cite[Section 4]{JM}), a contractive Gleason
solution for $\scr{H} (b)$ is a row contraction $X : \scr{H} (b) \otimes \C ^d \rightarrow \scr{H} (b)$ which obeys
$$ z(X^*f ) (z) = f(z) - f(0) ; \quad \quad f \in \scr{H} (b), \ z \in \B ^d, $$ and which is contractive in the sense that
$$ X X ^* \leq I  - k_0 ^b (k_0 ^b) ^*.$$   Analogously, a map $\mbf{b} : \H \rightarrow \scr{H} (b) \otimes \C ^d$ is called a contractive Gleason solution for $b \in \scr{S} _d (\H )$ if
$$ z \mbf{b} (z) = b(z) - b(0); \quad \quad z \in \B ^d, $$ and if it is contractive in the sense that
$$ \mbf{b} ^* \mbf{b} \leq I_\H - b(0) ^* b(0).$$ Observe that in the classical single-variable case, the unique contractive Gleason solutions for $\scr{H} (b)$ and $b$ are given by $X = S^*| _{\scr{H} (b)}$ and $\mbf{b} = S^* b$, where $S$ is the shift on $H^2 (\D)$. In contrast, as soon as $d>1$, contractive Gleason solutions for $\scr{H} (b)$ and $b$ are generally non-unique (but they can be parametrized in a natural way, see \cite[Section 4]{JM}).

Every contractive Gleason solution for $\scr{H} (b)$ is determined by a contractive Gleason solution for $b \in \scr{S} _d (\H)$: Given any contractive Gleason solution $X$ for $\scr{H} (b)$,
there is a contractive Gleason solution $\mbf{b}$ for $b$ so that
 \be X^* k_w ^b = w^* k_w ^b - \mbf{b} b(w ) ^*; \quad \quad w \in \B ^d, \quad \mbox{\cite[Section 4]{JM}}. \label{GSonto} \ee Any contractive Gleason solution $X$ for $\scr{H} (b)$ necessarily obeys: $$ k_w ^b = (I -X w^* ) ^{-1} k_0 ^b. $$

Remarkably, the free theory is, in several ways, simpler and more closely parallels the classical single variable theory. Any right free deBranges-Rovnyak space $\scr{H} ^R (B)$ for $B = (B^L, B^R )\in \scr{L} _d (\H ) \times \scr{R} _d (\H)$ is always invariant for $L ^* \otimes I _\H$, the adjoint of the left free shift (similarly $\scr{H} ^L (B)$ is invariant for $R^* \otimes I_\H$). Moreover, if one defines contractive (right) free Gleason solutions $\hat{X} ^R , \mbf{B} ^R$ for $\scr{H} ^R (B)$ and $B$ as in the commutative setting, then these are always unique and given by
\be (\hat{X} ^R) ^* := (L^* \otimes I_\H) | _{\scr{H} ^R (B) } ; \quad \quad \mbox{and} \quad \quad \mbf{B} ^R := (L^* \otimes I_\H) B^R. \label{uniqueFGS} \ee  (In the left case we obtain $\mbf{B} ^L = (R^*\otimes I_\H)  B^L$ and $(\hat{X} ^L) ^* := (R^* \otimes I_\H) | _{\scr{H} ^L (B)}$.)

Namely, a contractive Gleason solution for any right free deBranges-Rovnyak space $\scr{H} ^R (B)$ can be defined as a row-contraction $\hat{X} : \scr{H} ^R (B) \otimes \C ^d \rightarrow \scr{H} ^R (B)$ such that
$$ Z ((\hat{X} ^R) ^* F) (Z) = F(Z) - F _\emptyset; \quad \quad F \in \scr{H} ^R (B), $$
and which is contractive in the sense that
$$ \hat{X}^R (\hat{X} ^R) ^* \leq I - \hat{k} ^R _\emptyset (\hat{k} ^R _\emptyset ) ^*. $$ This definition is equivalent to
$ \hat{k} ^R _\emptyset = (I -\hat{X} ^R Z^* ) \hat{k} ^R _Z$, or,
$$ \hat{k} ^R _Z = (I - \hat{X}^R  Z^* ) ^{-1} \hat{k} ^R _\emptyset.$$ Similarly, a contractive Gleason solution for $B^R$ is a map
$\mbf{B} ^R : \H \rightarrow \scr{H} ^R (B) \otimes \C ^d$ which obeys
$$ Z \mbf{B} ^R (Z) = B^R (Z) - B ^R _\emptyset, $$ and which is contractive in the sense that
$$ (\mbf{B} ^R ) ^* \mbf{B} ^R \leq I - (B^R _\emptyset) ^* B^R _\emptyset. $$

\begin{remark} \label{uniqueGS}
Exactly as in the commutative setting, \cite[Theorem 4.4]{JM}, one can show that if $\mbf{B}$ is any contractive Gleason solution for $B^R$ then
$$ \hat{X}  ^* \hat{k} _W ^R :=  \hat{k} _W ^R W^* - \mbf{B} \bullet_R B(W) ^*, $$ defines a contractive Gleason solution for $\scr{H} ^R (B)$. The transfer function theory of \cite[see Remark 4.4]{Ball2006Fock}, shows that $\scr{H} ^R (B), B^R$ always have the unique contractive Gleason solutions given by the formulas (\ref{uniqueFGS}) above.
\end{remark}

\begin{prop} \label{freeGSform}
The unique contractive Gleason solution $\mbf{B} ^R : \H \rightarrow \scr{H} ^R (B ) \otimes \C ^d$ for $B^R$ is given by the formula
$$ \mbf{B} ^R = (L^* \otimes I_\H ) B^R = \rft \pi _B (L) ^* [I\otimes ] _B (I - B_\emptyset ). $$
\end{prop}

\begin{proof}
Write $B := B^R$ and let  $\mbf{A} '  := \rft \pi _B (L) ^* [I \otimes ] _B $. Then,
\ba Z \mbf{A} '  (Z) & = & Z  (\hat{k} _Z ^{R} \otimes I_d  ) ^* \mbf{A} ' \nn \\
& = &  \left( \rft \pi_B (L) Z ^* (\rft) ^* \hat{k} _Z ^R \right) ^* \rft [I \otimes ] _B. \nn \ea
Since $\rft := M ^R _{(I -B^R)} \rct = (M^R _{(I-B^R) ^{-1} } ) ^* \rct$, it follows that
$$ \rft [I \otimes ] _B = \hat{k} ^R _\emptyset (I - B_\emptyset ^* ) ^{-1}. $$
The bracketed term is then
\ba & &  \rft \sum _{j=1} ^d (Z^*) ^{j} \pi _B (L _j)(\rct) ^* (M^R _{(I-B)}) ^* \hat{k} ^R _Z   \nn \\
& = & \rft \sum _{j=1} ^d (Z^*) ^{j} \pi _B (L _j) \left( \sum _\alpha (Z^*) ^{\alpha ^T} \pi _B (L^\alpha ) [I\otimes ] _B \right) \bullet_R (I -B(Z) ) ^* \nn \\
& = & \rft \sum _{j=1} ^d \sum _\alpha (Z^*) ^{j\alpha ^T} \pi _B (L) ^{j\alpha} [I\otimes ] _B \bullet_R (I -B(Z)) ^*  \nn \\
& = &\rft \left( \sum _\alpha (Z^* ) ^{\alpha ^T} \pi _B (L) ^\alpha [I \otimes ] _B - [I \otimes ] _B \right) \bullet_R (I - B(Z)) ^* \nn \\
& = & \hat{k} _Z ^R - \hat{k} ^R _\emptyset (I - B_\emptyset ^* ) ^{-1}  (I - B(Z) ) ^*. \nn \ea
It follows that
\ba Z \mbf{A} ' (Z) & = &  \left( \hat{k} _Z ^R - \hat{k} ^R _\emptyset (I - B_\emptyset ^* ) ^{-1}  (I - B(Z) ) ^* \right) ^* \hat{k} ^R _\emptyset (I - B _\emptyset ^* ) ^{-1} \nn \\
& = & \hat{k} ^R  (Z, \emptyset) (I - B_\emptyset ^* ) ^{-1} - (I -B(Z) ) (I - B_\emptyset ) ^{-1} (I - B_\emptyset B_\emptyset ^* ) ( I - B_\emptyset ^* ) ^{-1} \nn \\
& = & (I - B(Z) B_\emptyset ^*) (I - B_\emptyset ^* ) ^{-1} - (I - B(Z) ) ( I - B_\emptyset ) ^{-1} (I - B_\emptyset B_\emptyset ^* ) ( I - B_\emptyset ^* ) ^{-1} \nn \\
& = & (I - B(Z) ) \bullet_R \hat{K} ^R (Z , \emptyset ) - \frac{1}{2} (I - B(Z) ) (H _\phi + H _\phi ^* ) \nn \\
& = & \frac{1}{2} ( I -B(Z) ) \bullet_R ( H _B (Z) - H _\emptyset ) \nn \\
& = & ( B(Z) - B_\emptyset ) (I - B_\emptyset ) ^{-1}. \nn \ea Hence $\mbf{A} := \mbf{A} ' (I - B _\emptyset)$ as defined above is a Gleason solution.

To see that $\mbf{A}$ is contractive note that if $B^R$ is a free lift of $b \in \scr{S} _d (\H)$,
\ba (\mbf{A}^* \mbf{A} ) & \leq & (I - B_\emptyset ^* ) [I \otimes ] _B ^* [ I \otimes ] _B (I - B_\emptyset ) \nn \\
& = & (I - b(0) ^* ) K ^b (0, 0) (I - b(0 ) ) \nn \\
& = & I - b(0) ^* b(0)  = I - B_\emptyset ^* B_\emptyset. \nn \ea By the uniqueness of the contractive Gleason solution for $B^R$, $\mbf{A} = \mbf{B} ^R = (L^* \otimes I_\H) B^R$ (Remark \ref{uniqueGS}).
\end{proof}

\begin{thm}{ (right free Clark Intertwining)}
Let $B = (B^L , B^R) \in \scr{L} _d (\H) \otimes \scr{R} _d (\H)$ be a transpose conjugate pair of free Schur multipliers.
The image of the adjoint of the row isometry $\pi _B (L)$ under the weighted right free Cauchy transform is a co-isometric perturbation of the restriction of $L^* \otimes I_\H$ to the (left free shift co-invariant) right deBranges-Rovnyak space $\scr{H} ^R (B)$:
$$ \rft \pi _B (L) ^* (\rft) ^* = L ^* \otimes I _\H | _{\scr{H} ^R (B) } + \mbf{B} ^R  (I - B_\emptyset ) ^{-1} (\hat{k} ^B _\emptyset) ^*, $$ where
$\mbf{B} ^R = (L^* \otimes I _\H ) B^R : \H \rightarrow \scr{H} ^R (B) \otimes \C ^d$ is the unique contractive Gleason solution for $B^R$.
\end{thm}
The left free Clark intertwining formulas are analogous and computed similarly. The proof below is formally very similar to the Clark intertwining result for the commutative setting of $b \in \scr{S} _d (\H)$, established in \cite[Theorem 4.16, Section 4]{JM}.

\begin{remark} \label{qeremark}
As shown in \cite{JM}, $\pi _B (L)$ is a Cuntz unitary (an onto row isometry) if and only if the image $b \in \scr{S} _d (\H )$ of $B$ under the Davidson-Pitts symmetrization (quotient) map is \emph{quasi-extreme}, \emph{i.e.} if and only if
$$ H^2 _ 0 (\mu _B) := \bigvee _{\n \in \N ^d ; \ \n \neq 0 } L^\n \otimes \H = \bigvee _{\n \in \N ^d } L^\n \otimes \H =: H^2  (\mu _B ) \subset F^2 (\mu _B ),$$ (at least in the case where $\dim{\H } < \infty$, see Remark \ref{NTkernprob}). In the several-variable theory, $H^2 (\mu _b)$ and $H^2 _0 (\mu _b)$ play the role of the classical analytic subspaces obtained as the closure of the analytic polynomials, and the closed linear span of the non-constant analytic monomials in $L^2 (\mu _b)$ when $d=1$ and $\mu _b$ is an AC measure.

If $\pi _B (L)$ is a Cuntz unitary, then the image of $\pi _B (L) ^*$ under the weighted right Cauchy transform
is a Cuntz unitary perturbation of the adjoint of the left free shift restricted to the right free deBranges-Rovnyak space $\scr{H} ^R (B)$. This is a direct generalization of Clark's classical result \cite{Clark1972} (Theorem \ref{uniperturb}), and we recover Clark's result in the single-variable, scalar-valued case. Given any unitary $U \in \L (\H)$, it is not difficult to check that $\scr{H} ^R (B U^* ) = \scr{H} ^R (B)$. Applying the above result to $BU^*$ for any such unitary $U$, yields the full $\mc{U} (\H)$-parameter family of co-isometric Clark-type perturbations of the restriction of the adjoint of the left free shift.
\end{remark}
\begin{proof}
    Let $B := B^R$. Calculate on formal kernel maps:
\ba (L^* \otimes I_\H) \hat{k} _W ^R & = & (L^* \otimes I_\H) \hat{k} _W - (L ^* \otimes I_\H)  M^R _B \hat{k} _W \bullet_R B(W) ^* \nn \\
& = &   \hat{k} _W W^* - (L^* \otimes I_\H)  M^R _B \hat{k} _W \bullet_R B(W) ^*. \nn \ea
Observe that in terms of the formal power series, each $L _j \otimes I _\H$ is a left multiplier so that $L ^* _j \otimes I _\H \hat{k} _W = \hat{k} _W W_j ^*$, and then calculate,
\ba (L ^* \otimes I_\H)  M^R _B \hat{k} _W \bullet_R B(W) ^*  & = & (L^* \otimes I_\H)  M^R _B (\hat{k} _W- I) \bullet_R B(W) ^* + (L^* \otimes I_\H) B  B(W) ^*  \nn \\
&  = &  M ^R _B \hat{k} _W W^* \bullet_R B(W) ^* + \mbf{B} B(W) ^*. \nn \ea
In summary this shows
\be (L^* \otimes I_\H)  \hat{k} ^R _W = \hat{k} ^R _W  W^* - \mbf{B} B(W) ^*, \label{ClarkA} \ee as expected, since $L^*  \otimes I_\H  | _{\scr{H} ^R (B) } $ is the unique contractive Gleason solution for $\scr{H} ^R (B)$.

Compare this to
\ba & & \rft \pi _B (L) ^* (\rft) ^* \hat{k} ^B _W \nn \\
& = & \rft \pi _B (L) ^* \left( \sum _{\alpha \neq \emptyset} (W^* ) ^{\alpha ^T} \pi _B (L) ^\alpha [I \otimes ] _B \right) \bullet_R (I - B(W) ) ^* + \rft \pi _B (L) ^* [I \otimes ] _B  (I -B(W) ^*) \nn \\
& =  & \rft \pi _B (L) ^* \left( \sum _{\alpha \neq \emptyset} (W^* ) ^{\alpha ^T} \pi _B (L) ^\alpha [I \otimes ] _B \right) \bullet _R (I - B(W) ) ^*
+ \mbf{B} (I - B_\emptyset ) ^{-1} (I -B(W)) ^*, \nn \ea where we have applied the previous proposition identifying $\mbf{B} = \mbf{B} ^R$ with $\mbf{A}$ to obtain the last line above.
It remains to calculate
\ba \rft \pi_B (L) ^* \sum _{\alpha \neq \emptyset} (W^*) ^{\alpha ^T} \pi _B (L ) ^\alpha [I\otimes ] _B & = &
\rft \bigoplus _j \sum _\beta (W^*) ^{\beta ^T} W^* _j \pi _B (L) ^\beta [I\otimes ] _B \nn \\
& = & ( M ^R _{(I -B) ^{-1} }) ^* \hat{\mc{C}} _R \left( * \circ T \circ [I\otimes ]_B ^* (I -W \pi _B (L) ^* ) ^{-1} \right) W^* \nn \\
& = &( M ^R _{(I -B) ^{-1} }) ^* \hat{K} ^R _W W^* \nn \\
& =& \hat{k} _W ^R W^* \bullet_R (I - B(W) ^* ) ^{-1}. \nn \ea
In summary,
\be \rft \pi _B (L) ^* (\rft) ^* \hat{k} ^B _W = \hat{k} ^R _W W^* + \mbf{B} (I -B_\emptyset) ^{-1} (I - B(W) ^* ). \label{ClarkB} \ee

Subtracting the expressions (\ref{ClarkA}) and (\ref{ClarkB}) yields:
\ba -(L^* \otimes I_\H)  \hat{k} _W ^R + \rft \pi _B (L) ^* (\rft) ^* \hat{k} _W ^R & =& \mbf{B}  B(W) ^* + \mbf{B} (I - B_\emptyset ) ^{-1}  (I - B(W) ) ^*. \nn \ea
If we define
$$ T := \mbf{B} (I-B_\emptyset) ^{-1} (\hat{k} ^R _\emptyset ) ^*  : \scr{H} ^R (B) \rightarrow \scr{H} ^R (B) \otimes \C ^d, $$ then on point evaluation maps,
\ba  T \hat{k} _W ^R &= & \mbf{B} (I - B_\emptyset )  ^{-1} (\hat{k} ^R _\emptyset ) ^*  \hat{k} _W ^R \nn \\
& = & \mbf{B}   (I-B_\emptyset ) ^{-1} \hat{k} ^R (\emptyset ; W ) \nn \\
& = & \mbf{B}  (I -B_\emptyset ) ^{-1} (I - B_\emptyset B(W) ^* ), \nn \ea  and then
\ba & & (T + (L^* \otimes I_\H) - \rft \pi _\phi (L) ^* (\rft) ^* ) \hat{k} _W ^R \nn \\
 &= &   \mbf{B}  ( (I -B_\emptyset ) ^{-1} (I - B_\emptyset B(W)) ^* - B(W) ^* - (I - B_\emptyset ) ^{-1} (I -B(W) ^* ). \nn \ea
The expression on the right evaluates to
\ba & &  (I-B_\emptyset ) ^{-1} ( I - B_\emptyset B(W)^* - (I-B_\emptyset ) B(W)^* - I + B(W)^* )  \nn \\
& = &  (I - B_\emptyset ) ^{-1} \left( I - B_\emptyset B(W)^* - B(W) ^* + B_\emptyset B(W) ^* - I + B(W)^* \right)  \nn \\
& = & 0, \nn \ea and this proves the Clark intertwining formulas.
\end{proof}

\section{Relationship between the free and commutative theories}

Recall the theory of non-commutative Aleksandrov-Clark measures for the commutative several-variable operator-valued Schur class $\scr{S} _d (\H)$ \cite{Jur2014AC,JM}. Let $\mc{S} = \mc{S} _d \subset \A _d = \A$ be the (norm-closed) symmetrized operator subspace:
$$ \mc{S} := \bigvee _{\n \in \N ^d} L^\n  = \bigvee _{z\in \B^d} (I -Lz^*) ^{-1}, $$ where $\bigvee$ denotes norm-closed linear span. Also recall that
$$ L ^\n := \sum _{\la (\alpha ) = \n } L^\alpha, $$ where $\la : (\F ^d , \cdot ) \rightarrow (\N ^d , + )$ is the unital letter-counting epimorphism.
As in the free theory of this paper, and as described in the introduction, there is a bijection between non-unital $b \in \scr{S} _d (\H)$, Herglotz-Schur class functions on $\B ^d$, and completely positive (AC) maps $\mu _b \in CP (\mc{S} ; \H )$, where $CP (\mc{S} ; \H )$ is the positive cone of completely positive maps of $\mc{S} + \mc{S} ^*$ into $\L (\H)$. In particular the Herglotz representation formula in this setting is
$$ H_b (z) = \mu _b \left( (I -Lz^*) ^{-1} (I + Lz^* ) \right) +i \im{H_b (0)}, $$ which is formally very similar to our free Herglotz representation formulas of Theorem \ref{FreeHform}.

The operator space $\mc{S}$, like the full free disk algebra $\mc{A}$, has the semi-Dirichlet property:
$$ \mc{S} ^* \mc{S} \subset (\mc{S} + \mc{S} ^* ) ^{- \| \cdot \| }, $$ so that one can again apply a GNS-type construction to obtain the \emph{Hardy space} of $\mu _b$, $H^2 (\mu _b)$, as the completion of the quotient of the algebraic tensor product $\mc{S} \otimes \H$ by vectors of zero length with respect to the pre-inner product: $$ \ip{ s_1 \otimes h _1 }{s_2 \otimes h_2 } _b := \ip{h_1}{\mu _b (s_1 ^* s_2) h_2} _\H. $$ If $\phi = \mu _B \in CP (\A ; \H )$ is a completely positive extension of $\mu _b$, that the Hardy space $H^2 (\mu _b )$ of $\mu _b$ embeds isometrically as a subspace $H^2 (\mu _B ) \simeq H^2 (\mu _b )$ of the free Hardy space $F^2 (\mu _B)$ of $\mu _B$.

\begin{cor} \label{freelifts}
      A free Schur class transpose-conjugate pair $B= (B^L , B^R) \in \scr{L} _d (\H ) \times \scr{R} _d (\H)$ is a pair of free lifts of $b \in \scr{S} _d (\H )$ if and only if $\mu _B \in CP (\mc{A} ; \H )$ is a completely positive extension of $\mu _b \in CP (\mc{S} ; \H )$
to the full free disk operator system $\mc{A} + \mc{A} ^*$.
\end{cor}

\begin{proof}
If $\mu _B$ extends $\mu _b$, then observe that $H_b (z)$ is obtained from $H_B ^L (Z)$ or $H_B ^R (Z)$ by substituting the commutative variable $z \in \B ^d$ in for $Z$. Hence $b(z)$ is obtained from $B ^L (Z), B^R (Z)$ in the same way.
This substitution amounts to applying the Davidson-Pitts symmetrization map which is known to be a completely contractive unital epimorphism of $L^\infty _d \otimes \L (\H )$ or $R^\infty _d \otimes \L (\H)$ onto $H^\infty _d \otimes \L (\H )$ \cite[Section 2]{DP-NP}.

Conversely, if $B ^L$ or $B^R$ is a free lift of $b$, then $H_B ^L (Z)$, or $H_B ^R (Z)$, evaluated at commutative $z$ must equal $H_b (z)$. By the Herglotz representation formulas for the free and commutative
Herglotz-Schur classes, it follows that
$$\mu _B ( (I -Lz^* ) ^{-1} ) = \mu _b ( (I -Lz^* ) ^{-1} ), $$ and this proves that $\mu _B | _{\mc{S} + \mc{S} ^* } = \mu _b$.
\end{proof}

Recall that any Schur class $b \in \scr{S} _d (\H)$, or $\mu _b \in CP (\mc{S} ; \H )$ are said to be \emph{quasi-extreme} if $H^2 (\mu _b) = H^2 _0 (\mu _b)$ where $H^2 _0 (\mu _b) \subset H^2 ( \mu _b)$, is the several-variable analogue of the closed linear span of the non-constant analytic monomials (see Remark \ref{qeremark}). This quasi-extreme property is a natural analogue of the single-variable Szeg\"{o} approximation property as described in the introduction, and it is related to extreme points of the Schur class \cite{JMqe}. See \cite{JM} for several equivalent characterizations of this property. The free theory of this paper provides yet another equivalent characterization.

\begin{cor}
    If a Schur class $b \in \scr{S} _d (\H )$ is quasi-extreme then it has a unique pair of transpose-conjugate free lifts $B=(B^L , B^R)  \in \scr{L} _d (\H ) \times \scr{R} _d (\H )$. The converse holds if $\H$ is finite dimensional.
\end{cor}

\begin{remark} \label{NTkernprob}
The converse holds provided that $b$ is quasi-extreme if and only if $\mu _b$ has a unique CP extension $\phi \in CP (\A ; \H )$. In \cite[Proposition 4.17]{JM} this was proven for all finite dimensional $\H$ (and for a large class of $b \in \scr{S} _d (\H)$ with $\H$ separable \cite[Proposition 4.14]{JM}). We expect $b$ is quasi-extreme if and only if $\mu _b$ has a unique extension, but the general result for separable $\H$ remains elusive at this time, see \cite[Remark 2.1]{JM}.
\end{remark}

\subsection{The Free and commutative deBranges-Rovnyak spaces}

As before, $B = (B^L , B^R)$ is a transpose-conjugate pair of free Schur class functions $B^L \in \scr{L } _d (\H )$, $B^R \in \scr{R} _d (\H )$.

\begin{lemma}
The map $C ^L _{H^2} : \scr{H} ^L (B) \rightarrow \scr{H} (b)$ defined by
$$C ^L _{H^2} (I _{F^2} - M ^L_B (M_B ^L) ^* ) \mbf{h} = (I _{H^2} - M_b M_b ^* ) \mbf{h}; \quad \quad \mbf{h} \in \H^2 _d \otimes \H, $$ is a co-isometry onto $\scr{H} (b)$
with initial space $$ [  (I _{F^2} - M_B ^{L} (M_B ^{L} ) ^* ) (H^2 _d \otimes \H ) ] ^{-\| \cdot \| _{\scr{H} ^L (B) } }. $$
\end{lemma}
An analogous co-isometry $C ^R _{H^2}$ is defined for the right free deBranges-Rovnyak space.
\begin{proof}
    Assume that $B=B^L$ and drop the superscript $L$, the same proof works for the right case. The proof follows from the definition of the deBranges-Rovnyak spaces as complementary range spaces: If $\mbf{h} \in H^2 _d \otimes \H$ then

\ba \| (I _{F^2} - M_B ^L (M_B ^L ) ^* ) \mbf{h} \| ^2 _{\scr{H} ^L (B) } & = & \| \sqrt{I _{F^2} - M_B ^L ( M_B ^L ) ^*} \mbf{h} \| ^2 _{ F ^2 } \nn \\
& = & \ip{(I _{F^2} - M_B ^L (M_B ^L ) ^*) \mbf{h} }{\mbf{h}} _{F^2} \nn \\
& = & \ip{ P_{H^2} (I _{F^2} - M_B ^L (M_B ^L) ^*) P _{H^2} \mbf{h} }{\mbf{h} } _{F^2} \nn \\
& = & \ip{(I - M_b M_b ^* ) \mbf{h}}{ \mbf{h} } _{H^2} \nn \\
& =& \| (I - M_b M_b ^*) \mbf{h} \| ^2 _{\scr{H} (b)}. \nn \ea
In the above we used that $H^2 _d \otimes \H$ is co-invariant for the left free multiplier $M_{B} ^L$ and that $(M_B ^L) ^* | _{H^2 _d \otimes \H} = M_b ^*$ since
$B= B^L$ is a left free lift of $b$.
\end{proof}

Recall that in the commutative theory, one defines Cauchy and weighted Cauchy transforms $\mc{C} _b : H^2 ( \mu _b) \rightarrow \scr{H} ^+ (H_b)$ and $\mc{F} _b : H^2 (\mu _b) \rightarrow \scr{H} (b)$ by
$$ \mc{C} _b ( (I -Lz^*) ^{-1} \otimes h + N_b ) = K_z ^b h, $$ and
$$ \mc{F} _b ( (I -Lz^*) ^{-1} \otimes h + N_b ) = k_z ^b (I-b(z) ^* ) ^{-1} h; \quad \quad \mc{F} _b = M _{(I-b)} \mc{C} _b,$$ and these define isometries onto the commutative Herglotz space $\scr{H} ^+ (H_b)$ and the deBranges-Rovnyak space $\scr{H} (b)$, respectively \cite[Section 2.7]{JM}.

\begin{prop} \label{freeabelCT}
Let $B = (B^L , B^R ) \in \scr{L} _d (\H) \times \scr{R} _d (\H)$ be a transpose-conjugate pair of free lifts of $b \in \scr{S} _d (\H)$. Then $\mc{F} _b = C ^R _{H^2} \rft P = C^L _{H^2} \lft P$ where $P$ projects $F^2 ( \mu _B)$ onto $H^2 ( \mu _B) \simeq H^2 (\mu _b)$ and $\lft, \rft$ are the left and right weighted free Cauchy transforms onto the left and right deBranges-Rovnyak spaces of $B$.
\end{prop}

\begin{proof}
We prove the right case, left is analogous.
For $z \in \B ^d$, we know that
$$\mc{F} _b (I - \pi _B (L) z^* ) ^{-1} [I\otimes ] _B = k_z ^b (I -b(z) ^* ) ^{-1} . $$
Compare the above to $$ \hat{\mc{F}} _R \left( * \circ T \circ (I -Z\pi _B (L) ^* )^ {-1} [I\otimes ] _B \right) = \hat{k} ^R _Z \bullet_R (I - B(Z) ^* ) ^{-1}. $$
In particular, applying $\hat{\mc{F} } _R$ to $(I - \pi _B (L) z^* ) ^{-1} [I\otimes ] _B$ amounts to substituting the commutative variables $z$ in for $Z$ in the above expression, where
$$\hat{k} ^R _z := (I - M_B ^R (M_B ^R )^*  ) \hat{k} _z (I-b(z) ^* )^{-1}, $$ and $$ \hat{k} _z := \sum _{\alpha} z^\alpha \hat{k} _\alpha = \sum _{\n \in \N ^d } z^\n \hat{k} _\n. $$ In the above, recall that
$\hat{k} _\alpha (Z) = Z^\alpha$, and we define $\hat{k} _\n := \sum _{\alpha ; \ \la (\alpha ) = \n } \hat{k} _\alpha.$ In particular, identifying $H^2 _d$ with symmetric Fock space, we have that $\hat{k} _z = k_z \in \L (\H , H^2 _d )$, so that
$$ \hat{\mc{F} } _R  (I - \pi _B (L) z^* ) ^{-1} [I\otimes ] _B = (I - M_B ^R (M_B ^R )^* ) k_z (I -b(z) ^* ) ^{-1} \in \ker{C^R _{H^2} } ^\perp, $$ and
$\mc{F} _b = C_{H^2} ^R \hat{\mc{F} } _R.$
\end{proof}

\begin{remark}
    It is also easy to check that that the range of $\rft | _{H^2 ( \mu _B )} $ is $(I - M_B ^R (M_B ^R) ^* ) (H^2 _d \otimes \H )$, the initial space of the co-isometry $C_{H^2} ^R $.
\end{remark}

\subsection{Transfer function realizations.} \label{TFsection}

As before, let $B = (B^L , B^R) \in \scr{L} _d (\H) \times \scr{R} _d (\H )$ be a transpose-conjugate pair of free $\L (\H) $-valued Schur class functions. Recall that by \cite{Ball2006Fock}, any $B ^L  \in \scr{L} ^\infty _d (\H )$ correpsonds uniquely to a (co-isometric, observable) \emph{canonical deBranges-Rovnyak colligation}:
$$ U ^R _{dBR} := \bbm A ^R _{dBR}  & B ^R _{dBR}  \\ C ^R _{dBR} & D ^R  _{dBR} \ebm : \bbm \scr{H} ^R (B) \\ \H \ebm \rightarrow \bbm \scr{H}
_R (B) \otimes \C ^d \\ \H \ebm, $$ where,
\ba A ^R _{dBR} & := & (L^* \otimes I_\H) | _{\scr{H} ^R (B) }, \quad \quad  B ^R _{dBR}  := L^* B ^R  \nn \\
 C ^R _{dBR} &:=&  (\hat{k} ^{R} _\emptyset) ^*, \quad \quad \mbox{and}, \quad \quad D ^R _{dBR} := B^R _\emptyset. \nn \ea 
The left Schur multiplier $B^L$ is then realized as the \emph{transfer function} of $U_{dBR}$ by the Schur complement formula
$$ B^L (Z) = D ^R _{dBR} + C ^R _{dBR} ( I - Z A ^R _{dBR} ) ^{-1} B ^R _{dBR},  $$ see \cite[Theorem 4.3]{Ball2006Fock}. Note that $A ^R _{dBR} = \hat{X} ^*$ is (the adjoint of) the unique contractive Gleason solution for $\scr{H} ^R (B)$ and
$B ^R _{dBR} = \mbf{B} ^R$ is our unique contractive Gleason solution for $B ^R$. This shows the (right) canonical deBranges-Rovnyak colligation for a left Schur class element $B ^L \in \scr{L} _d (\H )$ is expressed
in terms of operators on the right free deBranges-Rovnyak space $\scr{H} ^R (B)$, see \cite[Remark 4.5]{Ball2006Fock}. Similarly there is a canonical left colligation and transfer function realization for $B^R$ using the left free deBranges Rovnyak space $\scr{H} ^L (B )$.

In the commutative theory \cite{Ball2008,Ball2010} for Drury-Arveson space, any $b \in \scr{S} _d (\H)$ again always has canonical (weakly co-isometric, observable) deBranges-Rovnyak transfer function realizations and colligations, but these are generally non-unique. Namely, a contraction, $u _{dBR}$, is called a canonical deBranges-Rovynak colligation for $b$ if it can be written in block form as
$$ u_{dBR}  := \bbm a _{dBR} & b  _{dBR} \\ c _{dBR} & d _{dBR} \ebm : \bbm \scr{H} (b) \\ \H \ebm \rightarrow \bbm \scr{H} (b) \otimes \C ^d \\ \H \ebm,$$ where $d_{dBR}  := b(0)$, $c _{dBR} := (k_0 ^b ) ^*$, $b  _{dBR}$ is a contractive Gleason solution for $b$, and
$X := a_{dBR} ^*$ is a contractive Gleason solution for $\scr{H} (b)$. As proven \cite[Theorem 2.9, Theorem 2.10]{Ball2010}, given any contractive Gleason solution $X$ for $\scr{H} (b)$, there is a contractive Gleason
solution $\mbf{b}$ for $b$ so that the above colligation $u_{dBR}$ is a canonical deBranges-Rovnyak colligation (contractive, weakly co-isometric and observable). As in the free case, $b \in \scr{S} _d (\H )$ can be recovered from any such colligation $u _{dBR}$ with the transfer function formula: $$ b(z) = d _{dBR} + c _{dBR} (I -z \cdot a _{dBR} )^{-1} b _{dBR}.$$

In \cite[Section 4]{JM}, it was shown that there is a bijection between contractive Gleason solutions $\mbf{b} : \H \rightarrow \scr{H} (b) \otimes \C ^d$ for $b \in \scr{S} _d (\H)$ and row-contractive extensions $D \supseteq V^b$ of a certain canonical row partial isometry $V^b$ on the commutative Herglotz space $\scr{H} ^+ (H_b)$. Namely, the map $V^b : \scr{H} ^+ (H_b) \otimes \C ^d \rightarrow \scr{H} ^+ (H_b)$ defined by
$$ V^b z^* K_z ^b = K_z ^b - K_0 ^b; \quad \quad z \in \B ^d $$ defines a partial isometry with initial space $\bigvee _{z \in \B ^d} z^* K_z ^b \H$. If $V^b \subseteq D : \scr{H} ^+ (H_b) \otimes \C ^d \rightarrow \scr{H} ^+ (H_b)$ is any row-contractive extension of $V^b$ on $\scr{H} ^+ (H_b)$ (in the sense that $D (V^b) ^* V^b = V^b$) then the formula
 \be \mbf{b} [D] := U_b ^* D^* K_0 ^b (I -b(0) ), \label{GSform} \ee defines a contractive Gleason solution for $b$, and we let $X [D]$ denote the contractive Gleason solution for $\scr{H} (b)$ corresponding to $\mbf{b} [D]$ as in equation (\ref{GSonto}):
 $$ X[D] ^* k_z ^b = z^* k_z ^b - \mbf{b} [D] b(z) ^*; \quad \quad z \in \B ^d.$$ In the above, $U_b : \scr{H} (b) \rightarrow \scr{H} ^+ (H_b)$ is the onto isometric multiplier of multiplication by $(I -b(z) ) ^{-1}$. (We assume here that $b \in \scr{S} _d (\H )$ is non-unital,
\emph{i.e.}, $I -b(z)$ is invertible for $z \in \B ^d$ and $H_b (z)$ takes values in bounded operators.) Finally, we set
\be \label{ccform} u_{dBR} [D] := \bbm X [D] ^* & \mbf{b} [D] \\ (k_0 ^b) ^* & b(0) \ebm; \quad \quad D \supseteq V^b, \ D : \scr{H} ^+ (H_b) \otimes \C ^d \rightarrow \scr{H} ^+ (H_b). \ee
Theorem \ref{dBRtransfer} below will prove that any $u_{dBR} [D]$ is a canonical deBranges-Rovnyak colligation for $b$, and that the map $D \mapsto u_{dBR} [D]$ is surjective (neither of these facts is immediately obvious).

\begin{defn}
    Given any non-unital $b \in \scr{S} _d (\H )$, let $D \supseteq V^b$ be a row contractive extension of $V^b$ on $\scr{H} ^+ (H_b)$. Define the extension $\phi _D \in CP (\A ; \H )$ of $\mu _b \in CP (\mc{S} ; \H )$ by
$$ \phi _D (L ^\alpha ) := (K_0 ^b ) ^* D ^\alpha K_0 ^b \in \L (\H ). $$ Such an extension will be called a \emph{symmetric extension}.
\end{defn}
 The fact that $\phi _D \in CP (\mc{A}; \H )$ extends
$\mu _b \in CP (\mc{S} ; \H)$ follows from:

\begin{lemma}{ (\cite[Lemma 3.14]{JM})}
    A row contraction $D : \K \otimes \C ^d \rightarrow \K$ on $\K \supseteq \scr{H} ^+ (H_b)$ extends $V^b$, $V^b \subseteq D$, if and only if
$$ K_z ^b = (I -Dz^* ) ^{-1} K_0 ^b; \quad \quad z \in \B ^d. $$
\end{lemma}

In the case where $D = V^b$, $\phi _D$ is called the \emph{tight extension} of $\mu _b$. This was defined and studied in \cite{JM,Jur2014AC}. Since
each $\phi _D$ extends $\mu _b$, Corollary \ref{freelifts} implies that $\phi _D = \mu _{B[D]}$ for a unique transpose-conjugate pair $B[D] = (B[D] ^L , B[D] ^R ) \in \scr{L} _d (\H ) \times \scr{R} _d (\H )$.

\begin{lemma}\label{symext}
Let $D \subseteq V^b$ and let $\phi _D$ be the corresponding symmetric extension. Then $\pi _D (L) := \pi _{ \mu _{B(D)} } (L)$ is unitarily equivalent to the minimal isometric dilation of $D$ and
$H^2 (\phi _D) \simeq H^2 (\mu _b)$ is co-invariant for $\pi _D (L)$.
\end{lemma}
    This motivates the terminology \emph{symmetric extension} (the symmetric subspace $H^2 (\phi _D) \subseteq F^2 (\phi _D )$ is co-invariant for $\pi _D (L)$).
The proof is as in \cite[Proposition 3.7, Lemma 3.8]{JM}:
\begin{proof}
Let $\pi _D := \pi _{\phi _D}$ be the GNS representation of $\mc{A}$ on $F^2 (\phi _D)$. Then $T := \pi _D (L)$ is a row isometry and $H^2 (\phi _D) = \bigvee _{\n \in \N^d } T^\n [I \otimes ] _{\phi _D} \H$ is cyclic for $T$. Let $W$ be the minimal isometric dilation of $D$ on $\K _D \supseteq \scr{H} ^+ (H_b)$.
Since $W, L$ are row isometries, for any $\alpha , \beta \in \F ^d$,
$$ (L^\alpha ) ^* L^\beta  = \left\{ \begin{array}{cc} L^\la &  \beta = \alpha \la \\ (L^\la ) ^*  & \alpha = \beta \la \\
0 & \mbox{else,} \end{array} \right., $$ and similarly for $W$. Hence, assuming say that $\beta = \alpha \la$,
\ba \phi _D ((L^\alpha ) ^* L ^\beta ) & = & \phi _D (L ^\la ) \nn \\
& = & (K_0 ^b) ^* W ^\la K_0 ^b \nn \\
& = & (K_0 ^b) ^* (W^\alpha ) ^* W^\beta K_0 ^b. \nn \ea It follows that the map $\mc{C} _D : F^2 (\phi _D ) \rightarrow \K _D$ defined by
$$ \mc{C} _D  T ^\alpha [I \otimes ] _{\phi _D} := W ^\alpha K_0 ^b, $$ is an onto isometry (onto by minimality of $W$) which extends the Cauchy transform
$\mc{C} _b$ of $H^2 (\phi _D)$ onto $\scr{H} ^+ (H_b)$. In particular, $\mc{C} _D T ^\alpha = W^\alpha \mc{C} _D$, and since $\scr{H} ^+ (H_b )$ is co-invariant for $W$,
$W^* | _{\scr{H} ^+ (H_b) } = D^*$, it follows that $H^2 (\phi _D)$ is co-invariant for $T = \pi _D (L)$.
\end{proof}

This also yields the generalized Clark intertwining formulas:

\begin{thm}
Given any row contractive extension $D$ of  $V^b$ on $\scr{H} ^+ (H_b)$, the weighted Cauchy transform intertwines the co-isometry $\pi _D (L) ^*$
with a perturbation of the adjoint of the contractive Gleason solution $X(D)$ for $\scr{H} (b)$:
$$ \mc{F} _b \pi _D (L) ^* | _{H^2 (\mu _b)} = \left( X [D] ^* + \mbf{b} [D] (I - b(0) ) ^{-1} (k_0 ^b ) ^*  \right) \mc{F} _b. $$
\end{thm}

\begin{proof}
   The proof is exactly as in \cite[Section 4.15]{JM}, using that $H^2 (\mu _b) \simeq H^2 (\phi _D)$ is co-invariant for $\pi _D (L)$.
\end{proof}

Given any contractive extension $D \supseteq V^b$, and corresponding $\phi _D \in CP (\mc{A} ; \H)$ extending $\mu _b \in CP (\mc{S} ; \H )$ as above,
we write $U_{dBR} ^L [D] ,U_{dBR} ^R [D]$ for the canonical deBranges-Rovnyak colligations for the unique free Schur pair $B[D] = (B[D] ^L, B[D] ^R ) \in \scr{L} _d (\H ) \times \scr{R} _d (\H )$ corresponding to the extension $\phi _D$ by Corollary \ref{freelifts}.

\begin{thm} \label{dBRtransfer}
Given any non-unital $b \in \scr{S} _d (\H)$, let $B = (B^L , B^R) \in \scr{L} _d (\H) \times \scr{R} _d (\H)$ be a transpose-conjugate pair of free lifts of $b$.
Let $$ u_{dBR} = \bbm a_{dBR} & b_{dBR} \\ c_{dBR} & d_{dBR} \ebm := \bbm C_{H^2} \otimes I_d & 0 \\ 0 & I _\H \ebm  \ U_{dBR} \ \bbm C _{H^2} & 0 \\ 0 & I _\H \ebm ^* : \bbm \scr{H} (b) \\ \H \ebm \rightarrow \bbm \scr{H} (b) \otimes \C ^d \\ \H \ebm, $$ where $U _{dBR} = \bbm A_{dBR} & B_{dBR} \\ C_{dBR} & D_{dBR} \ebm $ is either the left canonical deBranges-Rovnyak colligation for $B^R$ or the right colligation for $B^L$. Then $u_{dBR} =: \mr{Ad} _{C_{H^2}} \circ U_{dBR}$ is a canonical deBranges-Rovnyak colligation for $b$ such that $b_{dBR} = C_{H ^2} B_{dBR}$ is a contractive Gleason solution for $b$, and  $a_{dBR}^* = C_{H^2} A_{dBR} ^* C_{H^2} ^*$ is the contractive Gleason solution for $\scr{H} (b)$ corresponding to $b_{dBR}$:
 $$ a_{dBR} k_w ^b = w^* k_w ^b - b_{dBR} b(w) ^*; \quad \quad w \in \B ^d.$$

This defines a surjective map, $\mr{Ad} _{C_{H^2}}$, from canonical deBranges-Rovnyak colligations of free lifts of $b$ onto canonical colligations for $b$. Every canonical colligation for $b$ has the form $u_{dBR} [D]$ for a unique contractive $D \supseteq V^b$ (see equation \ref{ccform}) and the map $\mr{Ad} _{C_{H^2}}$ is a bijection when restricted to canonical colligation pairs of the form $(U_{dBR} ^L [D], U_{dBR} ^R [D])$. A colligation pair $(U _{dBR} ^L , U_{dBR} ^R )$ corresponding to a free Schur class pair $B = (B^L , B^R )$ is in the inverse image of $u_{dBR} [D]$ under $\mr{Ad} _{C_{H^2}}$ if and only if the compression of $\pi _B (L) $ to $H^2 (\mu _B) \simeq H^2 ( \mu _b)$ is equal to
$\mc{C} _b ^* D \mc{C} _b$.
\end{thm}

\begin{remark} \label{QEcolligation}
By \cite[Theorem 4.17]{JM}, $b \in \scr{S} _d (\H )$ is quasi-extreme if and only if $V^b$ is a co-isometry, or equivalently if and only if $b$ has a unique
contractive (and necessarily extremal) Gleason solution $\mbf{b} = \mbf{b} [V^b ]$. Moreover, in this case $X = X [V^b ]$ is the unique contractive Gleason solution for $\scr{H} (b)$ and this solution is extremal.  It follows easily from this that $b \in \scr{S} _d (\H )$ is quasi-extreme if and only if $u_{dBR} = u_{dBR} [V^b]$ is the unique contractive canonical deBranges-Rovnyak colligation for $b$ and this colligation is an isometry.
\end{remark}

\begin{proof}
    Consider the right colligation case, let $B=B^R$ be any right free lift of $b$, we suppress the superscript $R$. Let $U_{dBR}$ be the unique canonical co-isometric deBranges-Rovnyak colligation for $B$. Given $B_{dBR} =B_{dBR} ^R = (L^* \otimes I_\H) B^R$ consider $b_{dBR} := C_{H^2} ^R B_{dBR}$.  This map $b_{dBR} : \H \rightarrow \scr{H} (b) \otimes \C ^d$ is contractive in the sense of a Gleason solution:
$$ b_{dBR} ^* b_{dBR} \leq B_{dBR} ^* B_{dBR} \leq I - b(0) ^* b(0). $$ Here, recall that $B^R _\emptyset = b(0)$. Moreover,
\ba b_{dBR} & = & C_{H^2} B_{dBR} \nn \\
& = & C_{H^2} \rft \pi _B (L) ^* [I \otimes ] _B (I - b(0) ) \nn \\
& = & \mc{F} _b P _{H^2 (\mu _b) } \pi _{B} (L) ^* [I \otimes ] _B (I - b(0) ), \nn \ea where we have applied Proposition \ref{freeabelCT} in the above. Define a row contraction $D$ on $\scr{H} ^+ (H_b)$ by
$$ D^* K_0 ^b = U_b b_{dBR} (I - b(0) ) ^{-1}.$$ If we can show that $D \supseteq V^b$, then equation (\ref{GSform}) and the results of \cite[Section 4]{JM} will imply that $b_{dBR} = \mbf{b} [D]$ is a contractive Gleason solution for $b$. By definition,
$$ D^* K_0 ^b = \mc{C} _b P_{H^2 (\mu _b)} \pi _B (L) ^* [I \otimes ] _B, $$ so that
$$ D = \mc{C} _b P _{H^2 (\mu _b) } \pi _{B} ( L ) ^* | _{H^2 (\phi _D) }.$$ Indeed, anything else in $H^2 (\mu _b)$ is spanned by elements of the form
$$ (Lz^* ) (I - Lz^* ) ^{-1} \otimes h, $$ and the action of $\pi _B (L) ^* $ on such elements is the same as  that of $\hat{V} := \mc{C} _b ^* (V^b)  ^* \mc{C} _b$. It follows that $D \supseteq V^b$ so that $b _{dBR} = \mbf{b} [D]$ is a contractive Gleason solution for $b$.

The corresponding Gleason solution $\hat{X} = A _{dBR} ^*$ obeys
$$ \hat{X} ^* \hat{k} ^R _W = k^R _W W^* - B_{dBR}  B(W) ^*, \quad \quad \hat{X} ^* = (L^* \otimes I_\H ) | _{\scr{H} ^R (B)}, $$ let $X := a_{dBR} ^* = C_{H^2} \hat{X} C_{H^2 } ^*$.
Then,
\ba X^* k_w ^b & = & C_{H^2} (L^* \otimes I_\H ) (I - M^R _B (M ^R _B) ^* ) k_w \nn \\
& = & C_{H^2} \left( (L^* \otimes I _\H) (k_w - k_0 ) - (L^* \otimes I_\H)  M^R _B k_w b(w) ^* \right) \nn \\
& = & C_{H^2} \left(  w^* k_w - (L^* \otimes I_\H )  M^R _B (k_w -k_0 ) b(w) ^* + (L^* \otimes I_\H) M^R _B k_0 b(w) ^* \right) \nn \\
&= & C _{H^2 }  w^* (I - M^R _B ( M ^R _B) ^* ) k_w + C_{H^2} (L^* \otimes I_\H ) M^R _B k_0 b(w ) ^* \nn \\
& = & w^* k_w ^b + C_{H^2 }  (L^* \otimes I_\H) M ^R _B k_0 b(w) ^* \nn  \\
& = & w^* k_w ^b + C_{H^2} (L^* \otimes I_\H ) B b(w) ^* \nn \\
& = & w^* k_w ^b + C_{H^2} B_{dBR} b(w) ^* \nn \\
& = & w^* k_w ^b + b_{dBR} b(w) ^*, \nn \ea
and this shows that $a_{dBR} ^* = X = X [D]$ is the contractive Gleason solution for $\scr{H} (b)$ corresponding to $b_{dBR} = \mbf{b} [D]$. Also note that
$c_{dBR} ^* = C_{H^2} \hat{k} _\emptyset ^R = k_0 ^b$. To prove that $u_{dBR} = u_{dBR} [D]$ as defined in the theorem statement is a canonical deBranges-Rovnyak colligation for $b$, it remains to show, by \cite[Theorem 2.9]{Ball2010}, that $u_{dBR}$ is contractive. Since $C_{H^2}$ is a contraction, this is clear, and we conclude that $\mr{Ad} _{C_{H^2}} (U_{dBR}) = u_{dBR} [D]$.

To prove that this map from canonical deBranges-Rovnyak colligations $U_{dBR} [D]$ for $B$ to deBranges-Rovnyak colligations for $b$ is onto, let $u_{dBR}$ be any canonical deBranges-Rovnyak colligation for $b$.  Since $b_{dBR}$ is a contractive Gleason solution for $b$, it follows that there is a contractive extension $D \supseteq V^b$ so that
$$ u_{dBR} = \bbm a_{dBR} & \mbf{b} [D] \\ (k_0 ^b) ^* & b(0) \ebm.$$
As described above, if $\phi _D \in CP (\A ; \H )$ is the completely positive  extension of $\mu _b$ corresponding to $D \supseteq V^b$, then $\phi _D = \mu _{B[D]}$ for a unique pair of free lifts $B[D] = (B[D] ^L , B[D] ^R)$.

By Proposition \ref{freeGSform}, the unique contractive Gleason solution for $\scr{H} ^R (B[D] )$ is
$$ \mbf{B} [D] ^R := \rft \pi _D (L) ^* [I\otimes ] _{B[D] } (I - B _\emptyset), $$ and as in the first part of the proof $\mbf{b}  := C_{H^2} ^R \mbf{B} [D] ^R$ is a contractive Gleason solution for $b$. Since
$H^2 (\mu _b) = H^2 (\mu _{B[D]})$ is co-invariant for $\pi _D (L)$, Proposition \ref{freeabelCT} implies that
$$\mbf{b} = \mc{F} _b \pi _D (L) ^* [I\otimes ] _b  (I - b(0) ). $$ Again, by the first part of the proof $\mbf{b} = \mbf{b} [D' ]$ where the contractive extension $D' \supseteq V^b$ is defined by
\ba (D' ) ^* K_0 ^b &  =  & U_b \mbf{b} (I - b(0)) ^{-1} \nn \\
& = & \mc{C} _b \pi _D (L)^* [I \otimes ] _b \nn \\
& = & D^* K_0 ^b.\quad \quad \mbox{(By Lemma \ref{symext}.)} \nn \ea This proves that $D ' = D$, and as in the first part of the proof, it follows that the image of $U_{dBR} ^{L,R} [D]$ under conjugation by $C_{H^2}$ is $u_{dBR} [D]$, and that this is a canonical colligation for $b$. Since both
$$ u_{dBR} [D] = \bbm X[D]^* & \mbf{b} [D] \\ (k_0 ^b) ^* & b(0) \ebm, \quad \mbox{and} \quad u_{dBR} = \bbm a_{dBR} & \mbf{b} [D] \\ (k_0 ^b )^* & b(0) \ebm, $$ are canonical colligations for $b$, the uniqueness result \cite[Corollary 2.9]{Ball2007trans}, implies that $a_{dBR} ^* = X[D]$, so that $u_{dBR} = u_{dBR} [D]$, and $Ad _{C_{H^2}}$ implements a bijection of canonical pairs $(U_{dBR} ^L [D], U_{dBR} ^R [D] )$ onto canonical colligations for $b$.
\end{proof}


\end{document}